\UseRawInputEncoding

\documentclass[a4paper,reqno]{amsart}

\usepackage{hyperref}

\usepackage{amssymb,amsxtra,amsfonts,dsfont}

\usepackage{graphicx} \usepackage{amsmath} \usepackage{amssymb}

\usepackage[svgnames]{xcolor}

\usepackage[color,matrix,arrow]{xy}

\input xy \xyoption{all}

\newtheorem{lem}{Lemma}[section]

\newtheorem{thm}[lem]{Theorem}

\newtheorem{pro}[lem]{Proposition}

\newtheorem{exa}[lem]{Example}

\newcommand{\slrw}[1]{\stackrel{#1}{\longrightarrow}}

\newcommand{\ttv}{\tau_{\vv}}

\newcommand{\lrw}{\longrightarrow}

\newcommand{\LL}{\Lambda}

\newcommand{\xa}{\alpha}
\newcommand{\xb}{\beta}
\newcommand{\xc}{\gamma}
\newcommand{\xd}{\delta}

\newcommand{\frD}{\mathfrak D}

\newcommand{\caM}{\mathcal M}

\newcommand{\RHom}{\mathbf{R}\strut\kern-.2em\operatorname{Hom}\nolimits}

\newcommand{\zZ}{{\mathbb Z}}
\newcommand{\qv}[2]{Q_{#1}[{#2}]}
\newcommand{\zzv}{{\mathbb Z_{\vv}}}
\newcommand{\zzs}[1]{{\mathbb Z\mid_{#1} }}

\newcommand{\tL}{\widetilde{\LL}}
\newcommand{\ttL}{\widetilde{\widetilde{\LL}}}

\newcommand{\dtL}{\Delta{\LL}}

\newcommand{\vs}{\mathfrak{s}}
\newcommand{\vt}{\mathfrak{t}}
\newcommand{\vu}{\mathfrak{u}}

\newcommand{\bbi}{\bar{i}}
\newcommand{\bbj}{\bar{j}}
\newcommand{\bxb}{\bar{\beta}}

\newcommand{\olL}{\overline{\LL}}

\newcommand{\GG}{\Gamma}
\newcommand{\tG}{\widetilde{\Gamma}}
\newcommand{\trho}{\tilde{\rho}}

\newcommand{\htM}{\wht{M}}

\newcommand{\htL}{\wht{\LL}}

\newcommand{\htQ}{\widehat{Q}}
\newcommand{\htG}{\widehat{\GG}}

\newcommand{\htU}{\bar{U}}
\newcommand{\bw}{\bar{w}}
\newcommand{\bp}{\bar{p}}
\newcommand{\wht}[1]{\widehat{#1}}
\newcommand{\wtt}[1]{\widetilde{#1}}

\newcommand{\wtb}[1]{\widetilde{\check{#1}}}
\newcommand{\tQ}{\widetilde{Q}}
\newcommand{\ttQ}{\widetilde{\widetilde{Q}}}
\newcommand{\ttG}{\widetilde{\widetilde{\GG}}}

\newcommand{\olQ}{\overline{Q}}

\newcommand{\olrho}{\overline{\rho}}

\newcommand{\om}[1]{\Omega^{#1}}

\newcommand{\sk}[1]{^{(#1)}}

\newcommand{\Ext}{\mathrm{Ext}}

\newcommand{\rMod}{\mathrm{Mod}\,}

\newcommand{\rmod}{\mathrm{mod}\,}

\newcommand{\dbA}[2]{\mathfrak D(\widetilde{\mathbf A}_{#1},{#2})}
\newcommand{\dbfA}[2]{\mathfrak D_f(\widetilde{\mathbf A}_{#1},{#2})}

\newcommand{\vv}{\diamond}

\newcommand{\id}{\mathrm{id}\,}

\newcommand{\Ob}{\mathrm{Ob}\,}
\newcommand{\Mor}{\mathrm{Mor}\,}

\newcommand{\caC}{\mathcal{C}}

\newcommand{\dmn}{\mathrm{dim}\,}

\newcommand{\arr}[2]{\begin{array}{#1}#2\end{array}}

\newcommand{\mat}[2]{\left(\begin{array}{#1}#2\end{array}\right)}

\newcommand{\eqqc}[2]{\begin{equation}\label{#1}#2\end{equation}}

\newcommand{\eqqcn}[2]{\[ #2 \]}
\newcommand{\eqqcnn}[2]{$ #2 $}

\newcommand{\Sl}{\mathrm{SL}}
\newcommand{\Gl}{\mathrm{GL}}

\begin{document}

\title{Multi-layer quivers and higher slice algebras}
\author[Guo, Hu and Luo]{Jin Yun Guo, Yanping Hu and Deren Luo}
\address{Jin Yun Guo, Yanping Hu\\ MOE-LCSM, School of Mathematics and Statistics\\ Hunan Normal University\\ Changsha, Hunan 410081, P. R. China\\ Deren Luo \\School of Mathematics \\ Hunan Institute of Science and Technology\\Yueyang, Hunan 414000, P. R. China}

\email{gjy@hunnu.edu.cn(Guo), 1124320790@qq.com(Hu), luoderen@126.com(Luo)}
\thanks{This work is partly supported by National Natural Science Foundation of China \#12071120, \#11901191, 
\#11771135, 
\#11671126, \#11271119, 
and the Construct Program of the Key Discipline in Hunan Province.
Guo appreciates Zhenxing Di for bringing his attention to the diagrams of  categories and for the references.}

\begin{abstract}
In this paper, we introduce multi-layer quiver and show how to construct an $(n+1)$-slice algebras of infinite type from an $n$-slice algebra of infinite type using the bound quivers.
This leads to  constructing $(n+1)$-slice algebras of infinite type as matrix algebra and as tensor algebra of an $n$-slice algebra and equivalences of their module categories as the module categories of diagram of some quiver of type $\widetilde{A}_{n+1}$.
\end{abstract}



\maketitle

\section{Introduction\label{int}}

In \cite{hio12}, Hershend, Iyama and Oppermann introduce $n$-hereditary algebras which include $n$-representation-finite algebras and $n$-representation-infinite algebras.
The $n$-representation-finite algebras are studied extensively \cite{i11,io11,hi11,hz11,c16,p17,jk19,di20}.
In \cite{i11}, Iyama introduces cone construction for certain $n$-complete algebras, in this way, he gets  a tower of higher representation-finite algebras and higher complete algebras starting with a given one.
Such construction is also obtained for a class of higher representation-finite algebras in \cite{gl16} using $n$-translation algebras.

It is natural to study the relationship between the higher representation-infinite algebras of different dimensions, and it is known that tensor product produces certain higher hereditary algebra of the same type \cite{hio12,hi11}.
The $n$-slice algebras are exactly the $n$-hereditary algebras whose  $(n+1)$-preprojective algebras are $(q,n+1)$-Koszul algebras, and $n$-slice algebras of infinite type are $n$-representation-infinite algebras\cite{gh21b}.
In this paper, we introduce the multi-layer quiver and apply this construction to produce $(n+1)$-slice algebra of infinite type from an $n$-slice algebra of infinite type.

For an $n$-slice algebra $\GG$, we relate it with three other algebras, its $(n+1)$-preprojective algebra $\Pi(\GG)$, its quadratic dual $\LL$ with bound quiver $Q$ and (twisted) trivial extension $\tL = \dtL$ of $\LL$, these four algebras are related by Koszul duality, trivial extension and higher preprojective constructions (See Theorem 6.6 of \cite{g20}).
The bound quiver $Q$ of $\LL$ is an  $n$-properly-graded quiver, which is a bound quiver with quadratic relation such that maximal bound paths have the same length $n$, and the bound quiver $Q^{\perp}$ of $\GG$ is an  $n$-slice(see Section \ref{pre} for the details).

Quivers have played very important role in the representation theory, especially the Gabriel quivers for the algebras and Auslander-Reiten quivers for the module categories \cite{ars95,ass06,r84}.
Recently, Iyama introduce higher Auslander-Reiten theory to study higher dimensional representation theory of algebras  \cite{i007,i07}, quivers are also important in higher dimensional representation theory in describing algebras and certain module categories \cite{i11,io11,gl16,gll19b}.
Quivers are used extensively in studying  $n$-slice algebras, especially for generalizing results of hereditary algebras \cite{gll19b,gx21}.

The ideal of the multi-layer quiver construction comes from the phenomenon of returning arrows.
Such phenomenon were observed in \cite{g11} for the McKay quivers when  embedding the finite subgroup $\Gl(n,\mathbb C) $ to $\Sl(n+1,\mathbb C) $, in \cite{g16} when constructing trivial extension for an $n$-translation algebra and in \cite{gyz14} for constructing one dimensional higher Koszul AS-regular algebra from a given one using trivial extension of self-injective algebra and Koszul duality.
It is observed that returning arrow quiver has the feature of 'increasing the dimension' by one in certain cases \cite{g11,gyz14}.

For an $n$-properly-graded quiver $Q$, let $\LL$ be the $n$-properly-graded algebra defined by $Q$,  we first construct the trivial extension $\tL=\dtL$ of $\LL$ with the returning arrow quiver $\tQ$ of $Q$ as its bound quiver.
Then construct the trivial extension $\ttL = \Delta \tL$ of $\tL$ and get a returning arrow quiver $\ttQ$ of $\tQ$.
Now construct an smash product $\ttL\#' k \mathbb Z^*$ with the universal covering quiver $\zZ_{\vv}\ttQ$, which is a stable $(n+1)$-translation quiver with $(n+1)$-translation $\tau_{\vv}$,  as its bound quiver.
The complete $\tau_{\vv}$-slices in $\zZ_{\vv}\ttQ$ are an $(n+1)$-properly-graded algebra whose quadratic dual is an $(n+1)$-slice algebra of infinite type.
We find one, denoted by $\htQ$ and called  multi-layer quiver of $Q$, which can be read out directly from $Q$ when  $Q$ is nicely-graded (See Section \ref{pre} for definition).
The algebra $\htL$ defined by the bound quiver $\htQ$ is an $(n+1)$-properly-graded algebra and its quadratic dual $\htG$ is a nicely-graded $(n+1)$-slice algebra.

With the multi-layer quivers, we see how the algebras $\GG$ and $\htG$ (respectively, $\LL$ and $\htL$) are related, as triangulate matrix algebras and as tensor algebras.
We show that their representation theory are related using the representation theory of diagram of quivers.
The diagrams of a small categories are used recently in representation theory  and related researches \cite{dlly22b,dllx21b,es17,gk22,m20}.
Using  the multi-layer quiver, we can describe the $(n+1)$-slice algebra $\htG$ as a triangulate matrix algebra with diagonal entries in $\GG$ (Theorem \ref{qdual}),  and as a tensor algebra of certain bimodules over $\GG^{n+1}$ (direct sum of $n+1$ copies of $\GG$) (Theorem \ref{multidualalgtensor}).
We also find the module category of $\htG$ is equivalent the module category of certain diagram of a quiver of type $\tilde{A}_{n+1}$ with only one arrow in the opposite direction, assigning to each vertex the category of $\GG$-modules.
We leave further study of the representation theory of the multi-layer algebras in the future.

For higher slice algebra $\GG$ of infinite type, using the multi-layer quiver construction, with $\GG(0) = \GG$ and $\GG(t+1) = \widehat{\GG(t)}$, we get an iterated way of constructing a tower of higher slice algebras of infinite type.
It is interesting to see that such construction is entirely different from the tower of higher representation-finite  algebras constructed in \cite{i11,gl16}.
We also show by examples that multi-layer construction is not close for $n$-slice algebra of finite type, and that tensor product construction in \cite{hio12} is not close for higher slice algebra of infinite type.

Though the idea is simple, we need tedious work dealing with bound quivers, which occupies a large part of the paper.

The paper is organized as follow.
In Section \ref{pre}, we recall $n$-properly-graded quiver and the returning arrow quiver for an $n$-properly-graded quiver.
Nicely-graded quiver and $n$-nicely-graded quiver are introduced.
In Subection \ref{stbq}, the $n$-slice algebra is introduced.
The first and the second $\zZ Q$ type construction of an $n$-properly-graded quiver and related algebras are recalled.
Concepts and constructions related to stable $n$-translation quivers needed in this paper are recalled.
Multi-layer quiver for an $n$-nicely-graded quiver is defined in Section \ref{mulq}, with easy consequences listed and a simple example presented.
In Section \ref{nicelygraded}, we show that nicely-grading is invariant under the higher $\zZ Q$-construction, and for a nicely-graded quiver $Q$, the first $\zZ Q$ type construction $\zzs{n-1}Q$ is  a connected component in the second $\zZ Q$ type construction $\zzv\tQ$.
In Section \ref{mqtslice},  the complete $\tau_{\vv}$-slice of $\zZ_{\vv}\ttQ$ for an $n$-nicely-graded quiver $Q$ is discussed, and the multi-layer quiver $\htQ$ is shown to be $(n+1)$-nicely-graded quiver by showing that it is a complete $\ttv$-slice in $\zZ_{\vv}\ttQ$.
In Section \ref{algebras}, we describe the multi-layer algebras as matrix algebras and as tensor algebras.
We also point out that the modules of the multi-layer algebra of an algebra can be regarded as the modules of the diagram of a quiver of type $\tilde{A}_{n+1}$ with only one arrow in the opposite direction.
In Section \ref{inftype}, we apply the multi-layer construction to discuss $(n+1)$-slice algebra constructed from an $n$-slice algebra of infinite type using multi-layer quiver.

This paper is an extended and generalized version of the corresponding results in '$\tau$-slice algebras of $n$-translation algebras and quasi $n$-Fano algebras,' arXiv:1707.01393, which is discontinued.

\section{preliminary\label{pre}}

\subsection{Algebras and bound quivers}
Throughout this paper, $k$ is a fixed field.
A bound quiver $Q= (Q_0,Q_1, \rho)$ is a triple of the vertex set $Q_0$, the arrow set $Q_1$ and a relation set $\rho$ which is a set of $k$-linear combinations of paths in $Q$.
In this paper the vertex set $Q_0$ may be infinite and the arrow set $Q_1$ is assumed to be locally finite.
The arrow set $Q_1$ is usually defined with two maps $\vs, \vt$ from $Q_1$ to $Q_0$ to assign an arrow $\alpha$ its starting vertex $\vs(\alpha)$ and its terminating vertex $\vt(\alpha)$.
We also write $\vs(p)$ for the starting vertex and $\vt(p)$ for its terminating vertex of a path $p$ in $Q$.

A bound quiver is related to an algebra over $Q$, that is, the quotient algebra $k Q/ (\rho)$ of the path algebra $kQ$ modulo the ideal $(\rho)$ generated by the relation set $\rho$.
A path $p$ in $Q$ is called a {\em bound path} if its image in $kQ/(\rho)$ is non-zero.

Let $\LL = \LL_0 + \LL_1+\cdots$ be a graded algebra over $k$ with $\LL_0$ direct sum of copies of $k$ such that $\LL$ is generated by $\LL_0$ and $\LL_1$.
Such algebra is determined by a bound quiver $Q= (Q_0,Q_1, \rho)$ \cite{g16}.

Let $S=\LL_0 = \bigoplus\limits_{i\in Q_0} k_i$, with $k_i \simeq k$ as algebras, and let $e_i$ be the image of the identity of $k$ under the canonical embedding of the $k_i$ into $S$.
Then $\{e_i\mid i \in Q_0\}$ is a complete set of orthogonal primitive idempotents in $S$ and $V= \LL_1 = \bigoplus\limits_{i,j \in Q_0 }e_j \LL_1 e_i$ as $S $-$ S $-bimodules.
Fix a basis $Q_1^{ij}$ of $e_j \LL_1 e_i$ for any pair $i, j\in Q_0$, take the elements of $Q_1^{ij}$ as arrows from $i$ to $j$, and let $Q_1= \cup_{(i,j)\in Q_0\times Q_0} Q_1^{ ij}.$
Let $Q_t$ be the set of the paths of length $t$ in $Q$ and let $kQ_t$ be the $k$-space spanned by $Q_t$, then $\LL_t = kQ_t/(\rho)_t$, here $(\rho)_t$ is the subspace of $(\rho)$ spanned by the linear combination of the paths of length $t$.
The relation set $\rho$ is a set of linear combinations of paths of length $\ge 2$.
We assume that it is normalized such that each element in $\rho$ is a linear combination of paths starting at the same vertex and ending at the same vertex, that is $\rho = \cup_{i,j\in Q_0} \rho(i,j)$,  where $\rho(i,j)$ is the set of relations which are a linearly combination of paths from $i$ to $j$.
We also assume that these paths are of the same length since we deal with graded algebra.

\medskip

Let $Q= (Q_0, Q_1, \rho)$ be an bound quiver and $\LL$ be the algebra given by the bound quiver $Q$.
The quiver $Q$ is called {\em acyclic} if $Q$ contains no oriented cycle, the algebra $\LL\simeq kQ/(\rho)$ is called {\em acyclic} if its quiver $Q$ is acyclic.
$Q$ is called {\em quadratic bound quiver} if relations are  quadratics, that is $\rho \subset kQ_2$.
If $Q= (Q_0, Q_1, \rho)$ is  a quadratic bound quiver, the quotient algebra $\LL = kQ/(\rho)$ is called a {\em quadratic algebra.}
The quadratic dual quiver $Q^{\perp} = (Q_0, Q_1, \rho^\perp)$ has the same vertex set $Q_0=Q_0$ as $Q$, has the dual basis  $Q_1 $ of $Q_1$ as arrow set, and has a basis $\rho^{\perp}$ in the orthogonal subspace $(k\rho)^{\perp} \subseteq k Q_2$ of $k \rho$ as relations when the dual space of $kQ_2$ is identified with $kQ_2$ itself (see \cite{g20}).
The algebra $\LL^{!,op} = k Q^{\perp} /(\rho^{\perp})$ is called the {\em quadratic dual} of $\LL$.

\subsection{The $n$-properly-graded quiver and nicely-graded quiver}
A bound quiver $Q$ is called {\em $n$-properly-graded} if all the maximal bound paths in $Q$ have the same length $n$.

A sequence  $w=(p_{0},\ldots,p_r) $ of paths in $Q$ with $l(p_t)>0$ for $t=1, \ldots, r-1$, such that $\vt(p_{2h+1})=\vt(p_{2h})$, $\vs(p_{2h+1}) = \vs(p_{2h+2})$ for $h = 0, \ldots, [\frac{r}{2}]$ is called a {\em walk} in $Q$ from $i=\vs(p_0)$ to $j=\vt(p_r)$, conventionally  $r$ is assumed to be even and $p_r$ is allowed to be a trivial path if $r$ is even, and write $i=\vs(w)$ and $j=\vt(w)$.
A walk $w$ is {\em cyclic} if $\vs(w)=\vt(w)$, in this case, we assume that $r=2r'+1$ is odd, and $l(p_t)>0$ for $t=0, \ldots, r$.
Define $\vu_{i,w}(j) = \sum_{h=0}^r (-1)^hl(p_h)$, and call it {\em the $(i,w)$-grade} of $j$ in $Q$. This is the difference of the number of arrows in the opposite directions passed when one goes from $i$ to $j$ along the walk $w$.

We have the following obvious result.
\begin{pro}\label{cyclezero}For any two vertices $i,j$ in $Q$, we have
\begin{enumerate}

\item\label{grdcyc} The $(i,w)$-grade of $j$ are independent of the walk $w$ if and only if for any cyclic walk $w'$ with $\vs(w')$ on $w$,  $\vu_{\vs(w'),w'}(\vt(w'))=0$.

\item\label{degcyc} If $Q$ is connected and $i$ is a vertex of $Q$, then for any vertex $j$ of $Q$, $\vu_{i,w}(j)$ is independent of $w$ if and only if for any cyclic walk with $\vs(w')=\vt(w')=i$, $\vu_{i,w'}(i)=0$.
\end{enumerate}
\end{pro}

A quiver with $\vu_{i,w}(i) =0$ for any vertex $i$ and any cyclic walk starting at $i$ is called a {\em nicely-graded quiver} in this paper.
In this case, if $Q$ is connected, for any $i,j \in Q_0$, write $\vu_{i}(j)= \vu_{i,w} (j) $ for any walk from $i$ to $j$, since $\vu_{i,w}(j)$ is independent of the choice of $w$.
Clearly, we have that if $\vu_{i_0,w}(i_0) =0$ for a vertex $i_0$ in the walk $w$, then $\vu_{i,w}(i) =0$ for any vertex $i$ in the walk $w$.

The following result follows from Proposition \ref{cyclezero}.
\begin{pro}\label{nicegrd}
Let $Q$ be a connected quiver.
Then $Q$ is nicely-graded if and only if there is a map $\vu$ from $Q_0$ to $\zZ$ such that $\vu(j)=\vu(i)+1$ if there is an arrow $\xa:i\to j$.
\end{pro}
So our definition of nicely-graded quiver coincide with the one given in \cite{g20}.
Clearly, a nicely-graded quiver is acyclic.
A bound quiver $Q= (Q_0,Q_1, \rho)$ is called an {\em $n$-nicely-graded} if  it is both nicely-graded and $n$-properly-graded.

\subsection{Stable $n$-translation quivers and related notions\label{stbq}}

Let $\caM$ be a maximal linearly independent set of maximal bound paths in $Q$, and let $\caM^*$ be a dual basis in the dual space of the space spanned by $\caM$.
Let $\tQ$ be the returning arrow quiver of $Q$, that is, the quiver obtained from $Q$ by adding an returning arrow $\beta_p: \vt(p) \to \vs(p)$ for each path $p$ in $\caM$.
Write $Q_{1,\caM} =\{\beta_p: \vt(p) \to \vs(p)\mid  p\in \caM\}$.
Let $\rho_{\caM} = \cup_{v,v'\in Q_0}\, \rho_{\caM}(v,v')$, where $\rho_{\caM}(v,v')$ the set of  linear combinations $\sum_{f\in \caM} c_f x_f \beta_f y_f$ of paths $x_f \beta_f y_f$ such that $c_f \in k$, $x_f, y_f$ are  paths with $\vt(x_f)=v'$, $\vs(y_f)=v$ satisfying $\sum_{f\in \caM} f^*( c_f y_f p x_f) =0$ for all paths $p$ from $\vt(x_f)$ to $\vs(y_f)$.

Clearly, we have that $\sum_{f\in \caM} c_f x_f f^* y_f=0$ if and only if $\sum_{f\in \caM} c_f x_f \beta_f y_f\in\rho_{\caM}$.
So $D\LL$ is a $\LL$-bimodule isomorphic to the $\LL$-bimodule with generator set $Q_{1,\caM}$ and relation set $\rho_{\caM}$.
And the trivial extension $\dtL$ of $\LL$ is the algebra with bound quiver $\tQ = (\tQ_0, \tQ_1,\trho)$, where $\tQ_0=Q_0, \tQ_1 = Q_1\cup Q_{1,\caM}$ and $\trho =\rho\cup \rho_{\caM}\cup \{\xb_p\xb_{p'}\mid p,p'\in \caM\}$, called a {\em returning arrow quiver} of $Q$ (see \cite{g20}).

The relations in $\rho_{\caM}$ have nice forms when the trivial extension is quadratic.

\begin{lem}\label{relzero} If $\dtL$ is quadratic, then the elements in $\rho_{\caM}$ are of the form $$\sum_{t} a_{\xa,t} \beta_{p_t} \xa +\sum_{t} b_{t,\xa'} \xa'\beta_{p'_t},$$ for some $a_{\xa,t}, b_{t,\xa'}\in k$.
\end{lem}

We also get a returning arrow quiver when take the relation set of a twisted trivial extension of $\LL$ for the bound quiver.
The following proposition follows directly from the definition of  returning arrow quiver.

\begin{pro}\label{boundpath_rq} Assume that $Q$ is $n$-properly-graded quiver.
Then each maximal bound path in a returning arrow quiver $\tQ$ is of length $n+1$ and contains at most one returning arrow.
\end{pro}

To study the multi-layer quiver of an $n$-nicely-graded quiver, we need some infinite stable $n$- and $(n+1)$-translation quivers.
An algebra defined by an $n$-properly-graded bound quiver will be called an {\em $n$-properly-graded
algebra} and an algebra defined by an $n$-nicely-graded bound quiver will be called an {\em $n$-nicely-graded algebra}.

The $n$-translation quivers and $n$-translation algebras are introduced in \cite{g16}.
The bound quivers of  graded self-injective algebras of Loewy length $n+2$ are exactly  the  stable $n$-translation quivers, with the $n$-translation $\tau$ induced by the Nakayama permutation.
Stable $n$-translation quiver is called stable bound quiver of Loewy length $n+2$ in \cite{g12}, each bound path of length $n+1$ goes from $\tau i$ to $i$ for some vertex $i$ in the quiver.

Given an $n$-properly-graded quiver $Q$ with a maximal linearly independent set $\caM$ of maximal bound paths in $Q$, let $\tQ$ be the returning arrow quiver of $Q$.
We have related to $Q$ some infinite bound  quivers $\zZ\mid _{n-1} Q$ (\cite{g16,g20}) and $\zZ_{\vv} \tQ$ (\cite{g12}, denoted by $\olQ$ there).

Recall that the quiver of $\zZ\mid _{n-1} Q$ is defined as the quiver  with the vertex set $$(\zZ\mid _{n-1} Q)_0 =\{(i , t)\mid  i\in Q_0, t \in \zZ\},$$ the arrow set $$\arr{lll}{(\zZ\mid _{n -1}Q)_1 & =& \zZ \times Q_1 \cup \zZ \times Q_{1,\caM} \\ & = &\{(\alpha,t): (i,t)\longrightarrow (j,t) \mid  \alpha:i\longrightarrow j \in Q_1, t \in \zZ\} \\ && \cup \{(\beta_p , t): (j, t) \longrightarrow (i, t+1) \mid  p\in \caM, \vs(p)=i,\vt(p)=j  \},}$$
and the relation set
$$\rho_{\zZ\mid _{n-1} Q} = \zZ \rho\cup \zZ \rho_{\caM} \cup \zZ\rho_0,$$ where $$\zZ \rho =  \{\sum_{*} a_* (\xa_s,t)\cdots (\xa_1,t) \mid a_*\in k, \sum_{*} a_* \xa_s \cdots \xa_1 \in \rho, t\in \zZ\},$$ $$\zZ \rho_{0} =  \{(\beta_{p'},t+1)(\beta_p ,t)\mid  \beta_{p'} \beta_{p}\in \rho_{\caM}, t\in \zZ\}$$ and $\zZ \rho_{\caM}$ is the relations for the bimodule $D \LL$ (see \cite{g20}), when $\dtL$ is quadratic, we have $$\arr{ll}{\zZ \rho_{\caM}= &\{ \sum_{s'} a_{s'} (\beta_{p'_{s'}},t) (\xa'_{s'}, t) + \sum_{s} b_s (\xa_s,t+1) (\beta_{p_s} ,t) \\ & \quad\mid  \sum_{s'} a_{s'} \beta_{p'_{s'}} \xa'_{s'} + \sum_{s} b_s \xa_s \beta_{p_s} \in \rho_{\caM} , t\in \zZ\}.}$$
The arrows $(\xa,t)$ arising from the arrows of the original quiver will be called {\em arrows of type $\xa$}, and the arrows in the set $\{(\beta_p , t): (j, t) \longrightarrow (i, t+1) \mid  p\in \caM, \vs(p)=i,\vt(p)=j  \}$ will be called {\em arrows of type $\beta$}.

We call $\zzs{n-1}Q$  {\em the first $\zZ Q$ type construction of $Q$}.
The following proposition follows directly from Proposition \ref{boundpath_rq} and the definition of $\zzs{n-1}Q$.
\begin{pro}\label{boundpath_zq} Assume that $Q$ is $n$-properly-graded quiver.
Then each bound path in $\zZ\mid _{n-1} Q$ is of length $n+1$ and contains at most one returning arrow.
\end{pro}

$\zZ\mid _{n-1} Q$ is an $n$-translation quiver with the $n$-translation defined by $\tau(i,m) =(i,m-1)$.
This quiver is realized as a bound quiver of a smash product and we have the following easy generalization of Proposition 5.5 of \cite{g16}.

\begin{pro}\label{0nap:extendible1}
Let $\LL$ be an $n$-properly-graded algebra with bound quiver $Q$.
Then  the smash product $\dtL\#  k \zZ^*$ is a self-injective algebra with bound quiver $\zZ\mid _{n-1} Q $, when $\dtL$ is graded by taking $\LL$ as degree zero component and  the elements in the dual basis of $\caM$ in $D\LL_n$ as degree $1$ generators.

$\zZ\mid _{n-1} Q $ is $(n+1)$-properly-graded if $Q$ is $n$-properly-graded.
\end{pro}
Write $\zzs{n-1}Q[t]$ for the full bound sub-quiver of  $\zzs{n-1}Q$ with vertex set $\{(j,t)\mid j\in Q_0\}$ for each integer $t$.

Starting with a stable $n$-translation quiver $\tQ=(\tQ_0,\tQ_1,\rho)$, we introduced a stable $n$-translation  quiver $\zZ_{\vv}\tQ = (\zZ_{\vv} {\tQ}_0, \zZ_{\vv}{\tQ}_1, \rho_{\vv}) $ in \cite{g12} (denoted by $\olQ$ and is called the  separated directed quiver of the stable bound quiver $\tQ$ there).
The vertex set is $$\zZ_{\vv}{\tQ}_0=\{(i,m) \mid  i \in \tQ_0, m \in \mathbb Z\}$$ and the arrow set is $$\zZ_{\vv}{\tQ}_1 = \{(\alpha, m):(i,m) \to (j,m+1)\mid \alpha: i \to j \in \tQ_1, m \in \mathbb Z \}.$$
If $p = \alpha_s \cdots \alpha_1 $ is a path in $\tQ$, write $(p,m) = (\alpha_s, m+s-1) \cdots (\alpha_1, m)$ for each $m \in \mathbb Z$.
A relation set is  $$\rho_{\vv} =\{(\zeta,m)\mid  \zeta \in \rho, m \in \mathbb Z \},$$ here $(\zeta,m) = \sum_{t} a_t (p_t,m)$ for each $\zeta = \sum_{t} a_t p_t \in \rho $.
$\zZ_{\vv}\tQ$ is a stable $n$-translation quiver with $n$-translation defined by $\tau_{\vv}(i,m) = (\tau i, m-n-1)$.
If $\tQ$ is a returning arrow quiver of $Q$,  $\zZ_{\vv} \tQ$ is also called {\em the second $\zZ Q$ type construction of $Q$}.

We have the following proposition directly from the definition.
\begin{pro}\label{0nap:niceg}
If $\tQ$ is an $(n+1)$-translation quiver, then $\mathbb Z_{\vv} \tQ $ is a nicely-graded $(n+1)$-translation quiver.
\end{pro}

The following Proposition is proven in Theorem 4.2 of \cite{g12}.

\begin{pro}\label{0nap:extendible2}
Let $\tL$ be a graded self-injective algebra defined by the stable $n$-translation quiver $\tQ$, then under the gradation of $\tL$ induced from the length of the paths, the smash product  $\tL\# ' k \mathbb Z^*$ is a graded self-injective algebra  given by the stable $n$-translation quiver $\mathbb Z_{\vv} \tQ $.
\end{pro}

The quiver $\zZ_{\vv}\tQ$ is a locally finite bound quiver if $\tQ$ is so, and it contains no oriented cyclic.
The quiver $\zZ_{\vv}\tQ$ has only finitely many $\tau$-orbit if $\tQ$ is finite.
If $\tQ$ is connected and $d$ is the maximal common divisor of the lengths of the minimal cycles in $\tQ$, then $\zZ_{\vv}\tQ$ has $d$ connected components.
In this case, the algebra $\tL$ defined by the bound quiver $\tQ$ is indecomposable, and $\tL\# ' k \mathbb Z^*$ is a direct sum of $d$ copies of isomorphic indecomposable algebras(see \cite{g12}).

Write $\zZ_{\vv}\tQ[l',l]$ for the full bound sub-quiver of $\zZ_{\vv}\tQ$ with vertex set $\{(j,r)\mid l'\le r \le l\}$ for integers $l' \le l$.

\medskip

Complete $\tau$-slice and $\tau$-slice algebras are introduced for stable $n$-translation quiver in \cite{g12}.
Let $\olQ=(\olQ_0, \olQ_1, \olrho)$ be an acyclic stable $n$-translation quiver with $n$-translation $\tau$, and assume that $\olQ$ has only finite many $\tau$-orbits.
Let $Q$ be a full sub-quiver of $\olQ$.
$Q$ is called a  {\em complete $\tau$-slice} of $\olQ$ if it has the following properties:
(a).for each vertex $v $ of $\olQ$, the intersection of the $\tau$-orbit of $v$ and the vertex set of $Q$ is a single-point set. (b). $Q$ is convex in $\olQ$.
Thus when normalizing the relations such that they are linear combinations of paths with the same starting vertex and the same ending vertex, then $$\rho = \{x = \sum_p a_p p \in \olrho\mid  \vs(p), \vt(p) \in Q_0 \} \subset \olrho.$$
We also call the bound quiver $Q = (Q_0, Q_1,\rho)$   a  {\em complete $\tau$-slice} of the bound quiver $\olQ$.

For an $n$-properly-graded quiver $Q$, $\zzs{n-1}Q[t]$ is a complete $\tau$-slice of $\zzs{n-1}Q$ which is isomorphic to $Q$, and $\zZ_{\vv}\tQ[0,n]$ is a complete $\tau$-slice of $\zZ_{\vv}\tQ$ for a finite stable $n$-translation quiver $\tQ$.

The following proposition is proved in Proposition 3.5 of \cite{gll19b}.
\begin{pro}\label{znq}
Let $\olQ$ be an acyclic stable $n$-translation quiver and  let  $Q=(Q_0,Q_1,\rho)$ be a complete $\tau$-slice of $\olQ$, then $\olQ \simeq \zZ\mid _{n-1} Q$ as bound quivers.
\end{pro}

So we can use $\zZ\mid _{n-1} Q$ for a stable $n$-translation quiver when a complete $\tau$-slice $Q$ is known.
If $Q$ is a nicely-graded complete $\tau$-slice, $Q$ is called {\em homogeneous} if its depth is $n$.

\medskip

The algebra $\LL$ whose bound quiver is a complete $\tau$-slice $Q$ in  $\olQ$ is called a {\em $\tau$-slice algebra} of the bound quiver $\olQ$.
If $Q$ is an $n$-properly-graded quiver, then it is a complete $\tau$-slice of $\olQ =\mathbb Z\mid _{n-1} Q$.
So we get immediately the following.

\begin{pro}\label{slice:sub}
A quiver $Q$ is an $n$-properly-graded quiver if and only if it is a complete $\tau$-slice of a stable $n$-translation quiver.
\end{pro}

We have a version of Proposition \ref{slice:sub} for algebras.
\begin{pro}\label{0nap:slicealgebra1} An algebra is $n$-properly-graded if and only if it is a $\tau$-slice algebra.\end{pro}

By Proposition 3.1 of \cite{gx21}, a $\tau$-slice algebra is both a quotient algebra and a subalgebra of the graded self-injective algebra $\olL$ defined by the bound quiver $\olQ$.

\medskip

When a complete $\tau$-slice of an stable $n$-translation quiver is quadratic, its quadratic dual  is called an {\em $n$-slice}.

In case that $\dtL$ of a $\tau$-slice algebra $\LL$ is an $n$-translation algebra, its quadratic dual $\GG= {\LL}^{!, op}$ is called {\em a $n$-slice algebra} (see \cite{g20}, it was called dual $\tau$-slice algebra in \cite{gx21}).
In \cite{gxl21}, $\GG$ is classified according to its $(n+1)$-preprojective algebra $\Pi(\GG)$, and $\GG$ is  of {\em finite type} if the dimension of $\Pi(\GG)$ is of finite dimension, and is of {\em infinite type} if $\Pi(\GG)$ of infinite dimension.
If $\GG$ is of infinite type, it is called of {\em tame type} if $\tQ$ is of finite complexity and of {\em wild type} if $\tQ$ is of infinite complexity.
If $\Pi(\GG)$ is of finite dimension, it is $(q, n+1)$-Koszul for some finite $q$ and is called of {\em finite type} in \cite{gxl21}.
If $\LL$ is $n$-nicely-graded, we also call $\GG$ a {\em nice $n$-slice algebra}.

\medskip

The $\tau$-hammocks are introduced in \cite{g12} for stable $n$-translation quivers, which are generalization of meshes in a translation quivers.
Let $\olQ= (\olQ_0,\olQ_1, \overline{\rho} )$ be a stable $n$-translation quiver with $n$-translation $\tau$ and let $\olL = k\olQ/(\olrho)$.
For each vertex  $i\in \olQ_0$,  the {\em $\tau$-hammock $H_i$ ending at $i$} is defined  as the quivers with the vertex set $$H_{i,0}=\{ (j,-t) \mid  j \in \olQ_0, \exists p\in \olQ_t, \vs(p)=j, \vt(p)=i,  0\neq p\in \olL \},$$ the arrow set $$\arr{ll}{ H_{i,1} = &\{ (\alpha , t): (j,-t-1) \longrightarrow (j', -t)\mid \alpha: j\to j'\in \olQ_1, \\ &\qquad \exists p\in \olQ_t, \vs(p)=j', \vt(p)=i, 0\neq p\xa \in \olL \}, }$$ and a hammock function $\mu_i: H_{i,0} \longrightarrow \zZ$ which is integral map on the vertices defined by $$\mu_i (j,-t)=\dmn_k e_i \olL_t e_j$$ for $(j,-t)\in H_{i,0}.$

Dually, for $i\in \olQ_0$, the  {\em $\tau$-hammock $H^i$ starting at $i$} is defined as the quivers with the vertex set $$H^i_0=\{ (j,t) \mid  j \in \olQ_0, \exists p\in \olQ_t,  \vs(p)=i, \vt(p)=j, 0\neq p\in \olL \}, $$ the arrow set
$$\arr{ll}{ H^i_1= & \{ (\alpha , t): (j,t) \longrightarrow (j', t+1)\mid \alpha: j\to j'\in \olQ_1, \\ &\qquad \exists p\in \olQ_t, \vs(p)=i, \vt(p)=j, 0\neq \alpha p\in \olL  \},}$$ and a hammock functions $\mu^i: H^i_0 \longrightarrow \zZ$ which is the integral maps on the vertices  defined by $$\mu^i (j,t)=\dmn_k e_j \olL_t e_i$$ for $(j,t)\in H^i_0.$
It is easy to see that $\iota: (j,t) \to (j, n+1+t) $ defines a bijective map from $H_{i,0}$ to $H^{\tau i}_0$ preserving the arrows, with the inverse  $\iota^{-1}: (j,t) \to (j, t-n+1) $.
We also call a vertex $j$ with $(j,t) \in H_{i,0}$ (respectively, $(j,t) \in H^i_{0}$) for some $t$ a vertex in the $\tau$-hammock $H_i$(respectively, $H^i$).

When $\olQ$ is nicely-graded, the hammocks $H^i$ and $H_i$ can be regarded as a sub-quiver of $\olQ$.

We  have the following observation.

\begin{lem}\label{tau-slice-hammock}
Let $Q$ be a complete $\tau$-slice in an acyclic stable $n$-translation quiver.

If $i$ is a sink in $Q$, then all the vertices of $H_i$ are in $Q$ except for $\tau i$.

If $i$ is a source in $Q$, then all the vertices of $H^i$ are in $Q$ except for $\tau^{-1} i$.
\end{lem}

If $i$ is a source of a complete $\tau$-slice $Q$ of $\olQ$, then all the vertices of the $\tau$-hammock $H^i$ except $\tau^{-1} i$ are in $Q$, and {\em the $\tau$-mutation $s_i^+ Q$} is obtained from $Q$ by removing the vertex $i$ and the arrows from it and  adding the vertex $\tau^{-1}i$ and the arrows from the vertices in $Q$ to it.
The $\tau$-mutation $s_{j}^- Q$ is defined dually  for a sink $j$ of $ Q$.
A $\tau$-mutation of a complete $\tau$-slice in $\olQ$ is again a complete $\tau$-slice in $\olQ$.
If $i$ is a source of $Q$, then $s_i^-s_i^+ Q = Q$, and if $i$ is a sink of $Q$, then $s_i^+s_i^- Q = Q$.

If $\LL$ is the $\tau$-slice algebra of the complete $\tau$-slice $Q$ in $\olQ$ and $\GG$ is the corresponding $n$-slice algebra.
The $\tau$-slice algebra of $s^\pm_i Q$ is called {\em the $\tau$-mutation} of $\LL$ at $i$ and is denoted by $s^\pm_i \LL$, and the $n$-slice algebra of $s^\pm_i Q^{\perp}$ is called {\em the $\tau$-mutation} of the $n$-slice algebra $\GG$ at $i$ and is denoted by $s^\pm_i \GG$.

A sequence $\{i_1, \ldots, i_m\}$ of the vertices of $Q$ is called {\em admissible source sequence} if $i_1$ is a source of $Q$ and $i_t$ is a source of $s^+_{i_{t-1}}\cdots s^+_{i_1} Q$ for $t=1,\cdots,m-1$.
An {\em admissible sink sequence} is defined dually.

The complete $\tau$-slices are related via $\tau$-mutations, we have the following refinement of Lemma 6.4 of \cite{g12}.
\begin{pro}\label{tauslicemut}
Let $Q$ and $Q'$ be two complete $\tau$-slice of a stable $n$-translation quiver $\olQ$ with finite $\tau$-orbits.
Then

\begin{enumerate}
\item There is an admissible source sequence $i_1,\ldots,i_r$ such that $Q'$ is isomorphic to $s^+_{i_r} \cdots s^+_{i_1} Q$.

\item There is an admissible sink  sequence $i_1,\ldots,i_r$ such that $Q'$ is isomorphic to $s^-_{i_r} \cdots s^-_{i_1} Q$.

\item There are $\tau$-mutations $s^*_1,\ldots,s^*_r$ such that $Q' = s^*_r \cdots s^*_1 Q$, where $s_t^* = s_{i_t}^+$ for some source $i_t$ of $s^*_{t-1} \cdots s^*_1 Q$ or $s_t^* = s_{i_t}^-$ for some sink $i_t$ of $s^*_{t-1} \cdots s^*_1 Q$.
\end{enumerate}
\end{pro}

\section{Multi-layer quiver\label{mulq}}

Let $Q$ be a finite $n$-nicely-graded quiver.
Take a source $i_0$ of $Q$, such that $\vu_{i_0}(j) \ge 0$ for all $j\in Q_0$.
For the convenience of describing multi-layer quiver, reindex the vertices in $Q$ as a pair $(i,\vu(i))$ with $\vu(i)= \vu_{i_0}(i)$ for each $i\in Q_0$.
Thus $(i_0,0)$ is a vertex of $Q$ and we call the maximal index $u_Q =\max\{\vu_{i_0} (i)\mid i\in Q_0\}$  the {\em depth of $Q$}.
Write an arrow $\alpha$ from $i$ to $j$ naturally as $(\xa,\vu(i)): (i,\vu(i))\to (j,\vu(j))$ and write a path $p$ from $i$ to $j$ as $(p,\vu(i))$ from $(i,\vu(i))$ to $(j,\vu(i)+l(p))$.
We will write the second index as $u$ when no confusion appears and call it the {\em relative degree} (with respect to $i_0$).

For a path $p$ of length $l$ starting at $i$ or a linear combination $x$ of paths with the same length $l$ starting vertex $i$, write $l(p)=l(x)=l$ and $\vu_{i_0} (p)= \vu_{i_0} (x)=\vu_{i_0}(i)  $.
We will write  $\vu(i)$, $\vu(p)$, $\vu(x)$, or just $u$ when no confusion appears.

For an $n$-nicely-graded quiver $Q$, we define the {\em multi-layer quiver $\wht{Q}$} over $Q$ as the bound quiver with vertex set $$\wht{Q}_0 = \{(i,\vu(i),\vu(i)+r)\mid (i,\vu(i))\in Q_0,r = 0,1,\ldots,n+1\},$$ the arrow set $$ \arr{lll}{\wht{Q}_1 &=&
\{(\xa, \vu(i),\vu(i)+r):(i, \vu(i),\vu(i)+r)\to (j, \vu(j),\vu(j)+r)\mid  \\ && \qquad (\xa,\vu(i)):(i,\vu(i)) \to (j, \vu(j))\in Q_1,r = 0,1,\ldots,n+1\} \\ &&
\cup \{(\gamma_i, \vu(i),\vu(i)+r) : (i,\vu(i),\vu(i)+r) \to (i,\vu(i),\vu(i)+r+1) \mid  \\ && \qquad  (i, \vu(i)) \in Q_0, r=0,1,\ldots,n \} \\ && \cup  \{(\beta_p,\vu(j),\vu(j)): (j, \vu(j),\vu(j)) \to (i,\vu(i),\vu(i)+n+1)\mid p: i \rightsquigarrow j\in \caM\},}$$
and a relation set
$$\arr{lll}{\widehat{\rho} &=& \{(x,\vu(x), \vu(x)+r)\mid x\in \rho, 0\le r \le n+1\} \\ && \cup \{(\gamma_{i},\vu(i),\vu(i)+r+1)(\gamma_{i},\vu(i),\vu(i)+r)\mid i\in Q_0, 0\le r \le n-1\} \\ && \cup\{(\xa,\vu(i),\vu(i)+r+1)( \gamma_{i},\vu(i),\vu(i)+r) \\ && \quad - ( \gamma_{j},\vu(j),\vu(j)+r) (\xa,\vu(i),\vu(i)+r)\\ && \qquad \mid \xa:i\to j\in Q_1, 0\le r\le n \} \\ && \cup\{\sum_{f\in \caM}c_f (x_f, \vu({x_f}), \vu({x_f})+n+1) (\beta_{f}, \vu({x_f})+n, \vu({x_f})+n) \\ && \quad\cdot (y_f, \vu({y_f}), \vu({y_f})) \\ && \qquad \mid \sum_{f^*\in \caM^*}c_f x_f f^* y_f\in \rho_{\caM}\}.}$$

The full subquiver $Q[r]$ with the vertex set $\{(i,u,u+r)\mid (i,u)\in Q_0\}$ is called the {\em $r$th floor} of $\htQ$, for $r=0, 1,\ldots,n+1$.
All the floors are isomorphic to $Q$ as bound quivers.
For each vertex $(j,t,t+r)$ of quiver $Q[r]$, there is an arrow $(\gamma_{j},t, t+r)$ from this vertex to $(j,t,t+r+1)$ in $Q[r+1]$ for $0\le r\le n$.
For each arrow $\beta_{p}: (\vt(p),\vu(\vs(p))+n) \to (\vs(p),\vu(\vs(p)) $ in $\tQ$, there is an arrow $(\beta_{p},\vu(\vs(p))+n,\vu(\vs(p))+n): (\vt(p),\vu(\vs(p))+n,\vu(\vs(p))+n) \to (\vs(p),\vu(\vs(p)),\vu(\vs(p))+n+1) $ from $Q[0]$ to $Q[n+1]$.

We have the following proposition.

\begin{pro}\label{quiver} If $Q$ is an  $n$-nicely-graded quiver then $\htQ$ is an $(n+1)$-nicely-graded quiver.
\end{pro}
\begin{proof}
Consider a path $\hat{p}$ in $\htQ$, if there are paths $p_1,\ldots,p_h$ in $Q$ satisfying $\vt(p_t)=\vs(p_{t+1})$ for $t=1, \ldots,h-1$ such that $$\arr{l}{\hat{p}= (p_h, \vu(\vs(p_h)),\vu(\vs(p_h))+h-1+r) (\gamma_{\vs(p_h)}, \vu(\vs(p_h)), \vu(\vs(p_h))+h-2+r) \\ \qquad \ldots (\gamma_{\vs(p_2)}, \vu(\vs(p_2)), \vu(\vs(p_2))+r) (p_1, \vu(\vs(p_1)), \vu(\vs(p_1))+r).}$$
Using the commutative relations concerning arrows of type $(\gamma_i, \vu(i),\vu(i)+r) $, we see that such path is a scalar of the path $$\arr{l}{(\gamma_{\vt(p_h)}, \vu(\vt(p_h)), \vu(\vt(p_h))+r+h-2) \ldots \\ \qquad (\gamma_{\vt(p_h)}, \vu(\vt(p_h)), \vu(\vt(p_h))+r)\cdot (p_h, \vu(\vs(p_h)),\vu(\vs(p_h))+r) \\ \qquad \ldots (p_1, \vu(\vs(p_1)), \vu(\vs(p_1))+r).}$$
So it is $0$ if either $h> 2$ or $\sum_{t}l(p_t) >n$.
Especially, it is zero in the quotient algebra if $l(\hat{p})> n+1$.

Otherwise, there are paths $p', p''$ in $Q$ and $q$ in $\caM$ such that $$\hat{p}= (p', \vu(\vs(p')),\vu(\vs(p'))+n+1) (\beta_q,\vu(\vt(q)),\vu(\vt(q))) (p'', \vu(\vs(p'')),\vu(\vs(p'')))$$ with $\vt(p'')=\vt(q)$ and $\vs(p')=\vs(q)$.
If  $l(\hat{p}) >n+1$, then $l(p')+l(p'') >n$, so we have that $$q^*(p''q'p') =0$$ for any path $q'$ in $Q$.
So $\hat{p}$ is in $\widehat{\rho}$ and $\hat{p}$ is zero in the quotient algebra.

This proves that bound paths has length at most $n+1$ in $\htQ$.

On the other hand, let $\hat{p}$ be a bound path of $\htQ$ with length $<n+1$.

If it is of the form $\hat{p} = (p, \vu(\vs(p)),\vu(\vs(p))+r)$ or $\hat{p} = (\gamma_{\vt(q)}, \vu(\vt(q)), \vu(\vt(q))+r)(q, \vu(\vs(q)),\vu(\vs(q))+r)$ with $q$ a path of length $l<n$.
Then $q$ is a bound path in $Q$, and we have paths $p',p''$ such that $p'q p''$ is a bound path of length $n$ in $Q$.
So $(p', \vu(\vs(p')),\vu(\vs(p'))+r+1)(\gamma_{\vt(q)}, \vu(\vt(q)), \vu(\vt(q))+r)(q, \vu(\vs(q)),\vu(\vs(q))+r)(p'', \vu(\vs(p'')),\vu(\vs(p''))+r)$
is a bound path of length $n+1$ in $\htQ$.
If for some $q\in \caM$, $\hat{p}= (p', \vu(\vs(p')),\vu(\vs(p'))+n+1) (\beta_q,\vu(\vt(q)),\vu(\vt(q))) (p'', \vu(\vs(p'')), \vu(\vs(p'')))$,
then there is a path $q'$ in $Q$, such that $l(p''q' p')= n$ and $q^*(p''q' p')\neq 0$, thus \small$(p', \vu(\vs(p')),\vu(\vs(p'))+n+1) (\beta_q,\vu(\vt(q)),\vu(\vt(q))) (p'', \vu(\vs(p'')), \vu(\vs(p''))) (q', \vu(\vs(q')), \vu(\vs(q')))$\normalsize is a bound path of length $n+1$ in $\htQ$.

This proves that $\htQ$ is an $(n+1)$-properly-graded quiver.

For a cyclic walk $\hat{w}$ with a vertex $(i,\vu(i),\vu(i)+r)$ in $\htQ$, its projection on the floor $Q[0]$ is a cyclic walk $w(0)$ in $Q[0]$ with a vertex $i$.
$\hat{w}$ is obtained from $w(0)$ by breaking $w(0)$ into pieces and put these pieces in different floors, and then connect them with arrows of type $(\gamma_i, \vu(i),\vu(i)+r) $.
The arrows of type $(\gamma_j, \vu(j),\vu(j)+h) $ appear in pairs to connect the pieces into a cyclic walk.
For each such pair of arrows of type $(\gamma_j, \vu(j),\vu(j)+h) $, when we go from a vertex $(i,\vu(i),\vu(i)+r)$
along the walk $\hat{w}$ to itself, the arrows are passed in opposite direction.
So $\vu_{(i,\vu(i),\vu(i)+r),\hat{w}}(i,\vu(i),\vu(i)+r) = 0$.

Thus $\htQ$ is nicely-graded and so $\htQ$ is an  $(n+1)$-nicely-graded quiver.
\end{proof}

\medskip

The following example shows how to construct a multi-layer quiver from an $1$-nicely-graded quiver $Q$ of type $\tilde{A}_1$.

\subsection*{Example}\label{ex1}
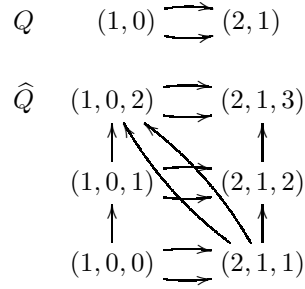
\[\arr{ll}{Q&\quad
\xymatrix@C=0.2cm@R0.5cm{
(1,0)\ar@/^/[rr]\ar@/_/[rr] && (2,1)
}\\
{}
\\
\htQ &
\xymatrix@C=0.2cm@R0.5cm{
(1,0,2)\ar@/^/[rr]\ar@/_/[rr] && (2,1,3)\\
(1,0,1)\ar@/^/[rr]\ar@/_/[rr]\ar[u] && (2,1,2)\ar[u]\\
(1,0,0)\ar@/^/[rr]\ar@/_/[rr]\ar[u] && (2,1,1)\ar[u]\ar@/^/[uull]\ar@/_/[uull]\\
}}
\]
This example appears in \cite{c16} as  the bound quiver of endomorphism algebra of some 2-hereditary tilting bundle $\mathcal T$ of some weighted surface.

\section{$n$-nicely-graded quivers\label{nicelygraded}}
We assume that $Q$ is acyclic in this section, in this case, $\zzs{n-1} Q$ is acyclic.

Note that for a finite stable $n$-translation quiver $\tQ$, $\zzv\tQ$ is always nicely-graded.
If $Q$ is not nicely-graded, $\zzs{n-1}Q$ is not nicely-graded, see Example \ref{nocycleho}.

Given an  $n$-nicely-graded quiver $Q= (Q_0,Q_1,\rho)$,  let $\tQ$ be its returning arrow quiver.
Since $Q$ is $n$-nicely-graded, it is acyclic, and the minimal cycles in  $\tQ$ have the same length $n+1$.
So $\zZ\mid _{n-1}Q$ is connected while $\zZ_{\vv} \tQ$ has $n+1$ isomorphic connected components, by Proposition 4.3 and 4.5  of \cite{g12}.
We are going to show that $\zZ\mid _{n-1} Q$ is nicely-graded when $Q$ is nicely-graded, so it is isomorphic to the connected components of $\zZ_{\vv} \tQ$.

By Lemma 2.1 of \cite{gw00}, we immediately get the following lemma.
\begin{lem}\label{two_bound_paths}
Let $\olQ$ be a stable $n$-translation quiver.
\begin{enumerate}
\item If $p_t$ and $p_{t+1}$ are two bound paths with $\vt(p_{t})=\vt(p_{t+1}) =\bbi $, then there are bound paths $p'_{t+1}$ and $p'_{t+2}$ with $\vs(p'_{t+1})=\vs(p'_{t+2}) =\tau \bbi$, $\vt(p'_{t+1})=\vs(p_{t})$ and $\vt(p'_{t+2})=\vs(p_{t+1})$ such that $$l(p_{t}) + l(p'_{t+1}) = n+1 =l(p_{t+1}) + l(p'_{t+2}).$$

\item If $p_{t}$ and $p_{t+1}$ are two bound paths with $\vs(p_{t}) = \vs(p_{t+1}) = \bbj$, then there are bound paths $p'_{t+1}$ and $p'_{t+2}$ with $\vt(p'_{t+1})=\vt(p'_{t+2}) = \tau^{-1}\bbj$, $\vs(p'_{t+1})=\vt(p_{t})$ and $\vs(p'_{t+2})=\vt(p_{t+1})$ such that $$l(p_{t}) + l(p'_{t+1}) = n+1 =l(p_{t+1}) + l(p'_{t+2}).$$
\end{enumerate}
\end{lem}
Using Lemma \ref{two_bound_paths}, we have the following lemma.
\begin{lem}\label{two_paths}
Let $\olQ$ be a stable $n$-translation quiver.
\begin{enumerate}
\item If $p_0=q_0p'_0$ and $p_1=q_1p'_1$ are two paths with $\vt(p_0)=\vt(p_1)$, such that $q_0$ and $q_1$ are both bound paths with $0 < l(q_0),l(q_1) \le n+1$, then there are bound paths $p''_0$ and $p''_1$ with $\vs(p''_0)=\vs(p''_1)$, $\vt(p''_0)=\vt(p'_0)$ and $\vt(p''_1)=\vt(p'_1)$ such that $$l(p_0) - l(p_1) = l(p'_0)-l(p''_0)+l(p''_1)-l(p'_1).$$

\item If $p_1=p'_1q_1$ and $p_2=p'_2q_2$ are two paths with $\vs(p_1)=\vs(p_2)$, such that $q_1$ and $q_2$ are both bound paths with $0 < l(q_1),l(q_2) \le n+1$, then there are bound paths $p''_1$ and $p''_2$ with $\vt(p''_1)=\vt(p''_2)$, $\vs(p''_1)=\vs(p'_1)$ and $\vs(p''_2)=\vs(p'_2)$ such that $$l(p_1) - l(p_2) = l(p_1')-l(p''_1)+l(p''_2)-l(p'_2).$$
\end{enumerate}
\end{lem}
\begin{proof}
We prove the first assertion.
Let $\bbi= \vt(p_0) = \vt(p_1)$.
By Lemma \ref{two_bound_paths}, there are bound paths $p''_{0}$ from $\tau \bbi$ to $\vs(q_{0})=\vt(p'_0)$ and  $p''_{1}$ from $\tau \bbi$ to $\vs(q_{1})=\vt(p'_1)$ in $\olQ$ such that the paths $ q_{0}p''_{0}$ and $q_{1} p''_{1}$ are both bound paths of length $n+1$ from $\tau \bbi$ to $\bbi$  in $\olQ$.
So \eqqcn{}{\arr{lll}{l(p_0) - l(p_1)& = &l(p'_0)+l(q_0) -(l(p'_1)-l(q_1))\\ &=&  l(p_0')+(n+1-l(p''_0)) -((n+1-l(p''_1))+l(p'_1))\\
&=&  l(p_0')-l(p''_0) +l(p''_1)-l(p'_1).}}
This proves the first assertion.
The second assertion follows similarly.
\end{proof}

Now we consider how the function $\vu$ changes under the $\tau$-mutations on a complete $\tau$-slice in an acyclic stable $n$-translation quiver  $\olQ$.
\begin{lem}\label{u-inv-taumut}
Let $Q$ be a complete $\tau$-slice in an acyclic stable $n$-translation quiver $\olQ$.
Assume that $\bbi,\bbi_0$ are vertices in $Q$ different from $\bbj$.
\begin{enumerate}
\item If $\bbj$ is a sink of $Q$ and $w$ is a walk from $\bbi_0$ to $\bbi$ with $\bbj$ a sink of $w$, then there is a walk $w'$ in $s^-_{\bbj} Q$ from $\bbi_0$ to $\bbi$ with $\tau \bbj$ a source such that $$\vu_{\bbi_0,w} (\bbi) = \vu_{\bbi_0,w'} (\bbi). $$

\item If $\bbj$ is a source of $Q$ and $w$ is a walk from $\bbi_0$ to $\bbi$ with $\bbj$ a source of $w$, then there is a walk $w'$ in $s^+_{\bbj} Q$ from $\bbi_0$ to $\bbi$ with $\tau^- \bbj$ a sink such that $$\vu_{\bbi_0,w} (\bbi) = \vu_{\bbi_0,w'} (\bbi). $$
\end{enumerate}
\end{lem}
\begin{proof}
We prove the first assertion, the second follows dually.

Let ${w}= (p_0, \ldots, p_{r})$ be an acyclic walk from $\bbi_0$ to $\bbi$ in $\olQ$.
Assume that $\vt(p_{2h})=\vt(p_{2h+1})$ is the first sink equals $\bbj$.
Let $p_{2h} = q_{2h}q'_{2h}$ and  $p_{2h+1} = q_{2h+1}q'_{2h+1}$ such that $q_{2h}$ and $q_{2h+1}$ are bound paths in $Q$, the lengths of $q_{2h}$ and $q_{2h+1}$ are both $\le n$.

By Lemma \ref{two_paths}, there are bound paths $q''_{2h}$ from $\tau \bbj$ to $\vs(q_{2h})$ and  $q''_{2h+1}$ from $\tau \bbj$ to $\vs(q_{2h+1})$ in $\olQ$ such that
\eqqc{equalaaa}{
l(q'_{2h}) -  l(q_{2h}'') +l(q_{2h+1}'') -l(q'_{2h+1}) \\ =l(p_{2h})-l(p_{2h+1}).}

Since $Q$ is a complete $\tau$-slice, all the vertices of $ q_{2h}q''_{2h}$ and $q_{2h+1} q''_{2h+1}$ are in $Q$ except for $\tau \bbj$.
So $q''_{2h+1}$ and $q''_{2h+1}$ are both in $s^-_{\bbj} Q$.

Let $p'_t =p_t$, for $t\le 2h-1$,  $p'_{2h}=q'_{2h}, p'_{2h+1}= q''_{2h}, p'_{2h+2} = q_{2h+1}'', p'_{2h+3}= q'_{2h+1}$, and $p'_{t+2} =p_t$ for $2h+2\le t \le r$, then
$$
w' = (p'_0, \ldots, p'_{2h-1}, p'_{2h}, p'_{2h+1}, p'_{2h+2}, p'_{2h+3}, p'_{2h+4}, \ldots, p'_{r+2})$$
is a walk in $\olQ$ which lies in the full subquiver $Q'$ with the vertex set obtained by  the union of those of $Q$ and of $\tau$-mutation $s^-_{\bbj} Q$, and $w'$ has one less sinks equal $\bbj$.
By equation \eqref{equalaaa}, we have $\vu_{\bbi_0,w'}(\bbi) =\vu_{\bbi_0,w}(\bbi)$.

Repeat the above process, we obtain a walk $w''$ in $Q'$ such that $\vu_{\bbi_0,w''}(\bbi) =\vu_{\bbi_0,w}(\bbi)$ and  no sink of $w''$ equals $\bbj$.
So $w''$ is in $s_{\bbj}^- Q$, as is required.
\end{proof}
So we have the following
\begin{pro}\label{nicely-inv-mut}
Let $Q$ be a complete  $\tau$-slice in an acyclic stable $n$-translation quiver $\olQ$.
If $Q$ is nicely-graded, its $\tau$-mutations are all nicely-graded.
\end{pro}
\begin{proof}
Let $\bbj$ be a source of $Q$ and let $Q'=s^+_{\bbj} Q$ be the $\tau$-mutation of $Q$ with respect to $\bbj$.
Then $\bbj'=\tau^{-1} \bbj$ is a sink of $Q'$ and $Q=s^-_{\bbj'} Q'$.
Let $w= (p_0, \ldots, p_{2r+1})$ be a cyclic walk from $\bbi$ to $\bbi$ in $Q'$.
If $\bbj'$ is not a sink of $w$, then $w$ is a walk in $Q$, and we have $\vu_{\bbi,w}(\bbi) =0$.
Otherwise, if $\bbj'$ is a sink of $w$ and $\bbi\neq \bbj$, then $w$ is not in $Q$, and by Lemma \ref{u-inv-taumut}, there is a walk $w'$ in $Q$ from $\bbi$ to $\bbi$ such that $$\vu_{\bbi,w}(\bbi)=\vu_{\bbi,w'}(\bbi)=0,$$ since $Q$ is nicely-graded.

If $\bbi=\bbj'$, then by the proof of Lemma \ref{u-inv-taumut}, we have a cyclic walk $w'$ of $Q$ such that for any vertex $\bbi'\neq \bbj'$, $$\vu_{\bbi',w}(\bbi')=\vu_{\bbi',w'}(\bbi')=0,$$
since $Q$ is nicely-graded.
So $$\vu_{\bbj',w}(\bbj')=\vu_{\bbi',w}(\bbi')=0.$$
This proves that $Q'$ is nicely-graded.

The case for $\bbj$ to be a sink can be  proved similarly.
\end{proof}

Clearly, the depth of a complete $\tau$-slice $Q$ in an $n$-translation quiver is no less than $n$.

If the depth of $Q$ is larger than $n$, then  for a vertex $i\in Q_0$, there is finite many $\bbj\in Q_0$ such that $\vu_{\bbi,w}(\bbj)$ take the maximal value.
Take such one, say $\bbj$, then $\bbj$ is a sink in $Q$, and $\tau \bbj$ is a source in $s^-_{\bbj} Q$.
There is a vertex $\bbj'$ such that there is an arrow $\xa$ from $\bbj'$ to $\bbj$ and there is a path $p$ of length $n$ in $s^-_{\bbj} Q$ such that $\xa p$ is a bound path in $\olQ$.
We may assume that $\bbi$ is in both $Q$ and $s^-_{\bbj} Q$.
Then we have that $\vu_{\bbi, w\xa p}(\tau \bbj)= \vu_{\bbi,w}( \bbj) - l(\xa p) = \vu_{\bbi,w} (\bbj) -1-n $ since both $Q$ and $s^-_{\bbj} Q$ are nicely-graded.
Inductively, we have the following proposition.
\begin{pro}\label{depth}
If $Q$ is a nicely-graded complete $\tau$-slice in a stable $n$-translation quiver $\olQ$, then there is a sequence $\xd_1,\cdots,\xd_r$ of $\tau$-mutations such that the depth of $\xd_r\cdots\xd_1 Q$ is $n$.
\end{pro}

Given a path $p$, we can write $p=p_1q$  such that $p_1$ is a maximal bound path.
Write $d(p) = d$ for the number of paths $p_1,\cdots, p_{d}$ such that $p= p_1\cdots p_{d}$ and $p_t$ is the maximal bound path in $p_t\cdots p_{d}$ for $1\le t\le d$.
If $p$ contains an arrow $\xb_q$ of type $\xb$, write $p= p_1 \xb_q p_0$, write $d'(p) =d (p_1 \xb_q) $  if $p_0$ contains no arrow of type $\xb$, and  write $d''(p) =d (\xb_q p_0) $  if $p_1$ contains no arrow of type $\xb$.

\medskip

Now we consider a walk of forms $V$ or $\Lambda$  in $\olQ$.
\begin{lem}\label{redV}
\begin{enumerate}
\item \label{redv1}
Let $w=(q_0,q_1)$ be a walk from $\bbi_0$ to $\bbi_1$ such that $q_0$ is path from $\bbi_0$ to  $\bbj$ and $q_1$ is a  path from $\bbi_1$ to  $\bbj$.
If the starting arrow of $q_0$ is not of type $\xb$, then there is a walk $w'=(q_0',q_1')$ or $w'=(p_0,p_1,q_0',q_1')$ from $\bbi_0$ to $\bbi_1$, such that $p_0,p_1$ contain no arrow of type $\xb$, the starting arrow of $q'_0$ is not of type $\xb$, $d(q'_1) + d'(q'_0)< d(q_1)+d'(q_0) $ and $$\vu_{\bbi_0, w'}(\bbi_1) =\vu_{\bbi_0,w}(\bbi_1).$$

\item\label{redv2} Let $w=(q_1,q_2)$ be a walk from $\bbj_1$ to $\bbj_2$ such that $q_1$ is a  path from $\bbi$ to  $\bbj_1$ and $q_2$ is a  path from $\bbi$ to  $\bbj_2$.
If the ending arrow of $q_1$ is not of type $\xb$, then there is a walk $w''=(q_1',q_2')$ or $w''=(p_1,p_2,q_1',q_2')$ from $\bbj_1$ to $\bbj_2$, such that $p_1,p_2$ contain no arrow of type $\xb$, the ending arrow of $q'_1$ is not of type $\xb$, $d(q'_2)+ d''(q'_1) < d(q_2)+ d''(q_1)$ and $$\vu_{\bbj_1,w''}(\bbj_2) =\vu_{\bbj_1,w}(\bbj_2).$$
\end{enumerate}
\end{lem}
\begin{proof}
We prove the first assertion  (\ref{redv1}).

Let $c'=d'(q_0)$ and $c =d(q_1)$, and assume that $q_0=q_{0,1}\bxb q_{0,0}$ and $q_{0,1}\bxb = q_{0}\sk{1}\cdots q_{0}\sk{c'}$ for bound paths $q_{0}\sk{1}, \cdots, q_{0}\sk{c'}$, the arrow $\bxb$ of type $\xb$,  and  $q_{1} = q_{1}\sk{1} \cdots q_{1}\sk{c}$ for bound paths  $q_{1}\sk{1}, \cdots, q_{1}\sk{c}$.

Set $\bp_{1}\sk{1}= q_{1}\sk{1}$.

Apply (2) of Lemma \ref{two_paths} successively on $q_{0}\sk{t}$ and $\bp_{1}\sk{t}$, we get  $\bp_{1}\sk{t+1}$ and $\bp_{0}\sk{t}$ such that $q_{0}\sk{t} \bp_{1}\sk{t+1}$ and $\bp_{1}\sk{t} \bp_{0}\sk{t}$ are bound paths of length $n+1$ in $\zzs{n-1}Q$.

If $l(\bp_{1}\sk{1})=l(q_1\sk{1}) = n+1$, then $l(\bp_0\sk{1})=0$ and $\bp_0\sk{1}$ is a trivial path, set  $q'_{0} = q_{0}\sk{2}\cdots q_{0}\sk{c'}q_{0,0}$ and  $q'_{1} = \bp_1\sk{2} q_{1}\sk{2} \cdots q_{1}\sk{c}$.
Then it is easy to check that $w'=(q_0',q_1')$ has the properties in (\ref{redv1}).

Assume $l(q_1\sk{1}) < n+1$.

If $l(q_{0}\sk{t'})<n+1$ for $1\le t' < t$ and  $l(q_{0}\sk{t})=n+1$, then $l(\bp_1\sk{t+1})=0$ and $\bp_1\sk{t+1}$ is a trivial path.
set  $q'_{0} =\bp_0\sk{1} \cdots \bp_{0}\sk{t}q_{0}\sk{t+1}\cdots q_{0}\sk{c'}q_{0,0}$ and  $q'_{1} =  q_{1}\sk{2} \cdots q_{1}\sk{c}$.
Then it is easy to check that $w'=(q_0',q_1')$ has the properties in (\ref{redv1}).

If $l(q_{0}\sk{t'})<n+1$ for $1\le t' \le c'$, if $\tau \vt(q_0\sk{c'}) = \vs(q_0)= \vs(\bp_1\sk{c'})$, set $q'_{0} = \bp_{0}\sk{1}\cdots \bp_{0}\sk{c'}$ and  $q'_{1} =  q_{1}\sk{2} \cdots q_{1}\sk{c}$.
Then it is easy to check that $w'=(q_0',q_1')$ has the properties in (\ref{redv1}).

Otherwise $\bp_1\sk{c'+1}$ contains no arrow of type $\xb$ since each bound path in $\zzs{n-1}Q$ contains  only one arrow of type $\xb$.
Set $p_0 = q_{0,0}, p_1 =\bp_1\sk{c'+1}$, $q'_{0} = \bp_{0}\sk{1}\cdots \bp_{0}\sk{c'}$ and $q'_{1} = q_{1}\sk{2} \cdots q_{1}\sk{c}$.

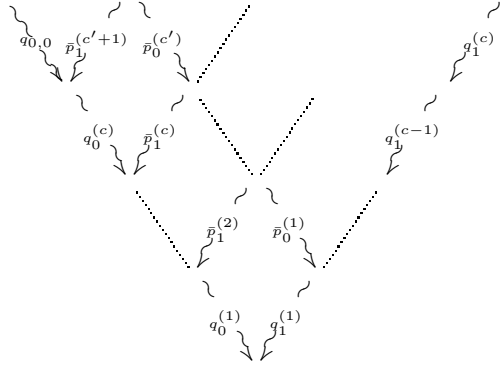
\tiny\eqqcn{}{\xymatrix@C=0.6cm@R1cm{
&&&\ar@{~>}[rd]|-{q_{0,0}} &&
\ar@[red]@{~>}[ld]|-{\bp_{1}\sk{c'+1}} \ar@[red]@{~>}[rd]|-{\bp_{0}\sk{c'}}&&
\ar@{..}[ld]|-{}&&
&&
\ar@{~>}[ld]|-{ \mbox{  }q_{1}\sk{c}}&&&\\
&&&&\ar@{~>}[rd]|-{q_{0}\sk{c}}&& \ar@[green]@{~>}[ld]|-{\bp_{1}\sk{c}} \ar@{..}[rd]&&
\ar@{..}[ld]|-{} &&\ar@{~>}[ld]|-{ q_{1}\sk{c-1}}&&\\
&&&&&\ar@{..}[rd]&& \ar@[green]@{~>}[ld]|-{\bp_{1}\sk{2}} \ar@[red]@{~>}[rd]|-{\bp_{0}\sk{1}}&&\ar@{..}[ld]&\\
&&&&&&\ar@{~>}[rd]|-{q_{0}\sk{1}}&&\ar@{~>}[ld]|-{ q_{1}\sk{1}}&&\\
&&&&&&&&&&
}}\normalsize

Then it is easy to check that $w'=(p_0, p_1, q_0',q_1')$ has the properties in (\ref{redv1}).

The second assertion is proved similarly.
\end{proof}

Now we prove that acyclic stable $n$-translation quiver with a nicely-graded complete $\tau$-slice is nicely-graded.

\begin{thm}\label{cychome}
If $Q$ is nicely-graded, then $\zzs{n-1}Q$ is also nicely-graded.
\end{thm}
\begin{proof}
By Proposition \ref{depth}, we may assume that $Q$ is homogeneous, that is, it is nicely-graded with depth $n$.
Let $\bbi_0$ be a source in $Q$ such that $0\le \vu_{\bbi_0}(\bbj) \le n$ for any $\bbj$ in $Q$.

Let $\bar{w} = (p_0, \ldots, p_{2r+1})$ be a cyclic walk from $\bbi$ in $\zzs{n-1} Q$, we prove that $\vu_{\bbi,\bar{w}}(\bbi) =0$.
By taking $\tau$-mutations, we may assume that $\vs(p_0)$ is in $Q$.

If the cyclic walk $\bar{w}$ contains no arrow of type $\xb$, then $\bar{w}$ is contained in a complete $\tau$-slice isomorphic to $Q$, which is nicely-graded, so we have $\vu_{\bbi,\bar{w}}(\bbi) =0$.

If $r=0$ and $\vs(p_0)=\tau(\vt(p_0))$ and both paths are bound paths, then both paths are of length $n+1$.
So $\vu_{\bbi,\bar{w}}(\bbi) =0$ for any vertex $\bbi$ on $\bar{w}$.

Otherwise, we show that there is a cyclic walk  $\bar{w}'$ without arrow of type $\xb$ from $\bbi'$ in $\zzs{n-1} Q$ such that  $\vu_{\bbi,\bar{w}}(\bbi) =\vu_{\bbi',\bar{w'}} (\bbi'). $

\medskip

Set $\bar{w}\sk{0}= \bar{w}$.

If $p_0 = p'_0 \bxb$, then  $p_{2r+1} = p'_{2r+1}\bxb'$, with $\bxb'$  an arrow of type $\xb$ since $Q$ is homogeneous and all arrows starting from $\vs(p_0)=\vs(p_{2r+1})$ are of type $\xb$.
Apply (2) of Lemma \ref{two_paths} for $p_{2r+1}, p_0$, with $\bxb'$ and $\bxb$ as the path $q_1$ and $q_2$ in Lemma \ref{two_paths}, then  there are paths $p''_{2r+1}$ and  $p''_{0}$ such that $\vt(p''_{2r+1})=\vt(p''_{0}) =\tau^{-1}\vs(p_0)$ and $$l(p_0)-l(p_{2r+1} )=l(p'_0)-l(p''_0)+l(p''_{2r+1})-l(p'_{2r+1}).$$
Since $p''_{2r+1}\bxb'$ and $p''_0\bxb$ are bound paths,   $p''_{2r+1}$ and $p''_0$ contains no arrow of type $\xb$.
Write $p\sk{1}_0 = p''_{2r+1}, p\sk{1}_1 = p''_0,q\sk{1}_0 = p'_0, q\sk{1}_1 = p_1$ and $ p\sk{1}_{2r+1} = p'_{2r+1}$ if $r>0$, and $q\sk{1}_1  = p'_{2r+1}$ if $r=0$, then $$\bw\sk{1} =(p\sk{1}_0,p\sk{1}_1,q\sk{1}_0, q\sk{1}_1, p_2, \cdots, p_{2r}, p\sk{1}_{2r+1})$$ is a cyclic walk from $\bbi\sk{1}=\tau^{-1}\bbi$, and $$\vu_{\bbi\sk{1},\bar{w}\sk{1}}(\bbi\sk{1})=\vu_{\bbi,\bar{w}}(\bbi).$$

\medskip

Note that if $r=0$ and  $\tau^{-1}\vs(p_0) \neq \vt(p_0)$, we have that $ p\sk{1}_{2r+1}=q\sk{1}_1 $.

\medskip

If $p_0$ does not start at an arrow of type $\xb$, let $q\sk{1}_0 =p_0, q\sk{1}_1=p_1$ and $p\sk{1}_{2r+1} =p_{2r+1}$, and   $$\bw\sk{1} =(q\sk{1}_0,q\sk{1}_1,p_2, \cdots, p_{2r}, p\sk{1}_{2r+1})$$ is a cyclic walk from $\bbi\sk{1}=\bbi$ and $$\vu_{\bbi\sk{1},\bar{w}\sk{1}}(\bbi\sk{1})=\vu_{\bbi,\bar{w}}(\bbi).$$

Note that we may assume that $\bbi\sk{1}=\vs(p_0\sk{1}) $ is in $Q$.

\medskip

Assume that we have a cyclic walk $$\bw\sk{l}=(p\sk{l}_0, \cdots, p\sk{l}_{2h\sk{l}+1}, q\sk{l}_{0},q\sk{l}_{1},p_{d\sk{l}}, \cdots, p_{2r},p\sk{1}_{2r+1})$$ from $\bbi\sk{l}$ in $\zzs{n-1} Q$, such that $p\sk{l}_t$ contains no arrow of type $\xb$ for $0 \le t \le 2h\sk{l}+1$ and $$\vu_{\bbi\sk{l},\bar{w}\sk{l}} (\bbi\sk{l}) =\vu_{\bbi,\bar{w}}(\bbi). $$

\medskip

If $q\sk{l}_{0}$ contains an arrow of type $\xb$, consider the walk $w=(q_0\sk{l}, q_1\sk{l})$.
Using induction reversely on $d'(q\sk{l}_{0})+d(q\sk{l}_{1})$, by applying (\ref{redv1}) of Lemma \ref{redV} starting with $w=(q_0\sk{l}, q_1\sk{l})$, we eventually get

(i). a walk $w' =(p_1',\cdots,p_{2m}',q_0',q_1')$  from the vertex $\vs(q_0\sk{l})$ to the vertex $\vs(q_1\sk{l})$ and $\vu_{\vs(q_0\sk{l}),w}(\vs(q_1\sk{l}))= \vu_{\vs(q_0\sk{l}),w'}(\vs(q_1\sk{l}))$, or

(ii). a walk $w'' =(p_1',\cdots,p_{2m}',q_0'')$ from the vertex $\vs(q_0\sk{l})$ to the vertex $\vs(q_1\sk{l})$ and $\vu_{\vs(q_0\sk{l}),w}(\vs(q_1\sk{l}))= \vu_{\vs(q_0\sk{l}),w''}(\vs(q_1\sk{l}))$.

Such that the paths $p_1',\cdots,p_{2m}'$,  and $q_0'$ contain no arrow of type $\xb$.
Set $h\sk{l+1}=h\sk{l}+m$, let $p_t\sk{l+1} =p_t\sk{l}$ for $0\le t\le 2h\sk{l}$, and  $p_{2h\sk{l}+t}\sk{l+1} =p'_t$ for $1\le t\le 2m$.

In case (i), let $q_0\sk{l+1} =q_0', q_1\sk{l+1}=q_1'$ and set $d\sk{l+1}=d\sk{l}$.

In case (ii),  let $q_0\sk{l+1} =p_{d\sk{l}}q_0'', q_1\sk{l+1} =p_{d\sk{l}+1}$ and set $d\sk{l+1}=d\sk{l}+2$.

Set $\bbi\sk{l+1}=\bbi\sk{l}$, now $$\bw\sk{l+1}=(p\sk{l+1}_0, \cdots, p\sk{l}_{2h\sk{l+1}-1}, q\sk{l+1}_{0},q\sk{l+1}_{1},p_{d\sk{l+1}}, \cdots, p_{2r},p\sk{1}_{2r+1})$$ is a cyclic walk from $\bbi\sk{l+1}$ in $\zzs{n-1} Q$.
One easily check that   $$\vu_{\bbi\sk{l+1},\bw\sk{l+1}}(\bbi\sk{l+1}) =\vu_{\bbi\sk{l},\bw\sk{l}} (\bbi\sk{l}). $$

\medskip

If $q\sk{l}_{0}$  contains no arrow of type $\xb$, then the arrow ending at $\vt(q\sk{l}_{0})$ is not of type $\xb$, and the ending arrow of $q_1\sk{l}$ is not of type $\xb$.
Let $q_2\sk{l}= p_{d\sk{l}})$ and consider the walk $w=(q_1\sk{l}, q_2\sk{l}$.
Similar to the above, using induction reversely on $d(q\sk{l}_{1})+d''(q\sk{l}_{2})$, by applying (\ref{redv2}) of Lemma \ref{redV} starting with $w=(q_1\sk{l}, q_2\sk{l})$, we eventually get

(a). a walk $w' =(p_1',\cdots,p_{2m}',q_1',q_2')$  from $\vt(q_1\sk{l})$ to $\vt(q_2\sk{l})$ and $\vu_{\vt(q_1\sk{l}),w'}(\vt(q_2\sk{l}))= \vu_{\vt(q_0\sk{l}),w}(\vt(q_2\sk{l}))$, or

(b). a walk $w'' =(p_1',\cdots,p_{2m}',q_1'')$ from the vertex $\vt(q_1\sk{l})$ to the vertex $\vt(q_2\sk{l})$ and $\vu_{\vt(q_1\sk{l}),w}(\vt(q_2\sk{l}))= \vu_{\vt(q_1\sk{l}),w''}(\vt(q_2\sk{l}))$.

Such that the paths $p_1',\cdots,p_{2m}'$,  and $q_1'$ contain no arrow of type $\xb$.
Set $h\sk{l+1}=h\sk{l}+m+1$, let $p_t\sk{l+1} =p_t\sk{l}$ for $0\le t\le 2h\sk{l}+11$, $p_{2h\sk{l}+2}\sk{l+1}=q_0\sk{l}$ and  $p_{2h\sk{l}+2+t}\sk{l+1} =p'_t$ for $1\le t\le 2m$, $p_{2h\sk{l+1}+1}\sk{l+1} =q'_1$.
Set $d\sk{l+1}=d\sk{l}+1$

In case (a), let $q_0\sk{l+1} =q_2', q_1\sk{l+1}=p_{d\sk{l}}$.

In case (b),  let $q_1\sk{l+1} =q_1' p_{d\sk{l}}, q_2\sk{l+1} =p_{d\sk{l}+1} (= p_{d\sk{l+1}})$.

Set $\bbi\sk{l+1}=\bbi\sk{l}$, now $$\bw\sk{l+1}=(p\sk{l+1}_0, \cdots, p\sk{l}_{2h\sk{l+1}-1}, q\sk{l+1}_{0},q\sk{l+1}_{1},p_{d\sk{l+1}}, \cdots, p_{2r},p\sk{1}_{2r+1})$$ is a cyclic walk from $\bbi\sk{l+1}$ in $\zzs{n-1} Q$.
One can check that   $$\vu_{\bbi\sk{l+1},\bw\sk{l+1}}(\bbi\sk{l+1}) =\vu_{\bbi\sk{l},\bw\sk{l}} (\bbi\sk{l}). $$

\medskip

If $d\sk{l+1} >2r+1$, then $$\bw\sk{l+1}=(p\sk{l+1}_0, \cdots, p\sk{l+1}_{2h\sk{l+1}-1}, p\sk{l+1}_{2h\sk{l+1}},p\sk{1}_{2r+1}),$$ this is case (i), $p\sk{1}_{2r+1}=q_1' ={p'}_{2r+1}\sk{l+1}$ and $p_0\sk{l+1}$ are paths start at $\bbi\sk{l+1}$.

Note that the function $\vu$ of $\zzs{n-1}Q$ coincides on the complete $\tau$-slice $Q$ with the one defined on $Q$.
If $\bbi_0$ is a source of $Q$, then  $0 \le \vu_{\bbi_0}(\bbi\sk{l+1}) < n+1 $ since $Q$ is homogeneous and $\bbi\sk{l+1}=\vs(p_0\sk{l+1})$ is in $Q$.
Since the path $p_{0} $ contains no arrow of type $\xb$, we have $0 \le \vu_{\bbi_0}(\bbi')< n+1 $ for any vertex $\bbi'$ in $p_0\sk{l+1}$ and the vertex $\bbi' = \vt(p_0\sk{l+1})=\vt(p_1\sk{l+1})$.
Similarly, we have that $0 \le \vu_{\bbi_0}(\vs(p_t\sk{l+1})) <n+1 $ and $0 \le \vu_{\bbi_0}(\vt(p_t\sk{l+1}))< n+1 $ for $t=1,\ldots,2h\sk{l+1}$.
Especially $0 \le \vu_{\bbi_0}(\vt(p\sk{1}_{2r+1})) = \vu_{\bbi_0}(\vt(p\sk{1}_{2r}))< n+1 $, so $\vt(p\sk{1}_{2r+1})\in Q_0$ and  $p\sk{1}_{2r+1}$ is a path in $Q$ since $Q$ is convex.

This proves that $\bw\sk{l+1}$ is a walk in $Q$ and we have that $$\vu_{\bbi\sk{l+1}),\bw\sk{l+1}}(\bbi\sk{l+1}) =0.$$

Now the theorem follows from induction on $r-d\sk{l}$.

\end{proof}

As a corollary, we have the following proposition.

\begin{thm}\label{cmpsmashp}
If $Q$ is an $n$-nicely-graded quiver, then $\zZ\mid _{n-1} Q$ is isomorphic to a connect component of  $\zZ_{\vv} \tQ$.
\end{thm}
\begin{proof}

Use the reindexing for $Q$.
Since $\zzs{n-1}Q$ is nicely-graded, the function $\vu_{(v,t), \bw } (v', t')$﹛is independent of the walk $\bw$,  so we extend  $\vu_{v} (v')$ of $Q$ uniquely to $\zzs{n-1}Q$ as $\vu_{(v,0)}(v',t')$.
﹛
Clearly, choose $v_0=(i_0,0)$ in $Q$ and write $\vu(v,t) = \vu(i,u,t) = \vu_{(i_0,0,0)} (i,u,t) $ on $\zzs{n-1}Q$ for $v=(i,u)$ in $Q$.

It is easy to see that the map $\Phi_{(i_0,0)}: \zZ\mid _{n-1} Q_0 \to \zZ_{\vv} \tQ_0$ defined by $$\Phi_{(i_0,0)}(v',m)= (v',(n+1)m+\vu_{(i_0,0)}(i,u))$$ for $v'=(i,u)$ induces an isomorphism from the quiver $\zZ\mid _{n-1} Q$ to a connected components of $\zZ_{\vv} \tQ$ preserving the relations.
\end{proof}

By choosing $i_0$ differently in a path of length $n+1$ in $\tQ$ in the above theorem, we see $\zZ\mid _{n-1}Q$ is isomorphic to each of the connected components of $\zZ_{\vv} \tQ$.
$\zzv \tQ$ has $n+1$ connected components isomorphic to $\zzs{n-1}Q$.
Thus $\zzv \tQ$ is also an $(n+1)$-properly-graded quiver.

\medskip

In the rest of the paper, we use the same notation for a vertex in $\zzs{n-1}Q$ as its image in $\zzv\tQ$ under the map $\Phi_{v_0}$ for some $v_0= (i_0,0)$ in $Q$.
So a vertex $(v,t)$ in $\zzs{n-1}Q$ will be written as $(v, \vu_{(v_0,0)} (v,t) )$, or $(i,\vu_{i_0}(i), \vu_{(v_0,0)}(v,t))$ for $v=(i,\vu_{i_0}(i))$ in $Q$.

Since  $\zZ_{\vv}\tQ$ is always acyclic, Theorem \ref{cmpsmashp} does not hold when $Q$ contains an oriented cycle.
The following example shows that it does not hold when $Q$ is not  nicely-graded.

\begin{exa}\label{nocycleho}
{\em
In the bound quiver \eqqcnn{}{Q: 
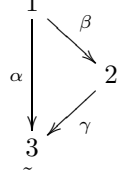
\xymatrix@C=0.2cm@R0.5cm{
1\ar[dd]_{\xa}\ar[rrd]^{\xb} \\
&&2\ar[lld]^{\gamma}\\
3}} with relation the path $\gamma\xb$ of length $2$.
This is an acyclic quiver of type $\tilde{A}_2$ of $3$ vertices, and it is impossible to take an nicely-graded orientation.
All the maximal bound paths in this quiver are the arrows, and hence are paths  of length $1$.
Its returning arrow quiver is  \eqqcn{}{\tQ: 
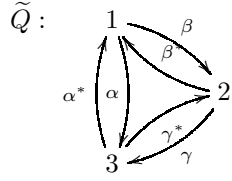
\xymatrix@C=0.2cm@R0.5cm{
1\ar@/^/[dd]_{\xa}\ar@/^/[rrrd]^{\xb} \\
&&&2\ar@/^/[llld]^{\gamma}\ar@/^/[lllu]_{\xb^*}&\\
3\ar@/^/[uu]^{\xa^*}\ar@/^/[rrru]_{\gamma^*}} }
with relation $\{\xa\xa^* - \gamma\gamma^*, \xa^*\xa - \xb^*\xb, \gamma^*\gamma- \xb\xb^*, \gamma^*\xa,\xb\xa^*, \gamma\xb, \xa\xb^*, \xb^*\gamma^*,\xa^*\gamma   \}$.

We have
 \eqqcn{}{\zzs{0}Q =\zZ Q: 
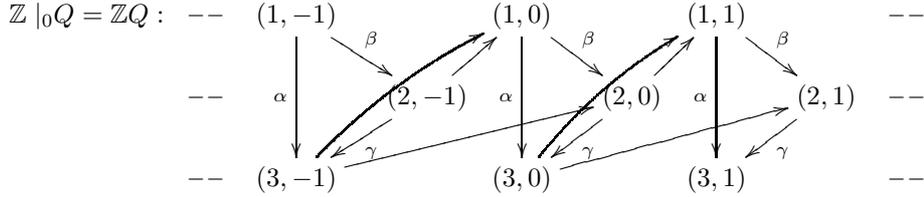
\xymatrix@C=0.1cm@R0.5cm{
&&\ar@{--}[ll]& (1,-1)\ar[dd]_{\xa}\ar[rrd]^{\xb}&&&  (1,0)\ar[dd]_{\xa}\ar[rrd]^{\xb} &&& (1,1)\ar[dd]_{\xa}\ar[rrd]^{\xb}&&& \ar@{--}[rr]&&&\\
&&\ar@{--}[ll]& &&(2,-1)\ar[lld]^{\gamma} \ar[ur]& &&(2,0)\ar[lld]^{\gamma} \ar[ur]& &&(2,1)\ar[lld]^{\gamma}& \ar@{--}[rr]&&&\\
&&\ar@{--}[ll]& (3,-1)\ar@/^/[uurrr]\ar[urrrrr] &&& (3,0)\ar@/^/[uurrr]\ar[urrrrr]&&& (3,1)&&& \ar@{--}[rr]&&&
}}
Its complete $\tau$-slices are isomorphic to $Q$ as bound quivers.
On the other hand, we have
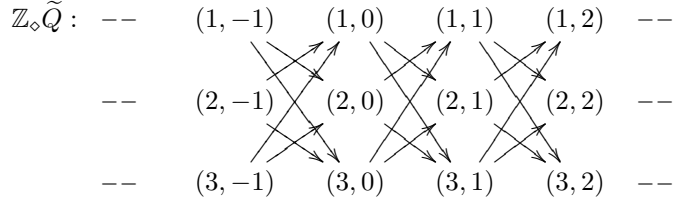
\eqqcn{}{\zzv\tQ: \xymatrix@C=0.1cm@R0.5cm{
&&\ar@{--}[ll]& &(1,-1)\ar[drr]\ar[ddrr]& &(1,0)\ar[drr]\ar[ddrr] & &(1,1) \ar[drr] \ar[ddrr] &&(1,2) &\ar@{--}[rr]&&\\
&&\ar@{--}[ll]& &(2,-1)\ar[drr]\ar[urr]& &(2,0)\ar[drr]\ar[urr] & &(2,1) \ar[drr] \ar[urr] &&(2,2) &\ar@{--}[rr]&&\\
&&\ar@{--}[ll]& &(3,-1)\ar[urr]\ar[uurr]& &(3,0)\ar[urr]\ar[uurr] & &(3,1) \ar[urr] \ar[uurr] &&(3,2) &\ar@{--}[rr]&&\\
}}

So $\zzs{n-1}Q$ is not a connected component of $\zzv\tQ$.
We have a homogeneous $\tau$-slice of $\zzv\tQ$
\eqqcn{}{\zzv\tQ[0,1]: 
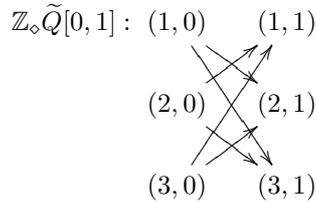
\xymatrix@C=0.1cm@R0.5cm{
(1,0)\ar[drr]\ar[ddrr] & &(1,1) \\
(2,0)\ar[drr]\ar[urr] & &(2,1)  \\
(3,0)\ar[urr]\ar[uurr] & &(3,1) \\
}}
It is a $1$-nicely-graded quiver of type $\tilde{A}_5$.
}\end{exa}

\section{Complete $\tau$-slices in $\zzs{n-1}Q$ and in $\zzv\ttQ$\label{mqtslice}}

From now on we assume that $Q$ is an $n$-nicely-graded quiver.
By Proposition \ref{slice:sub}, $Q$ is a complete $\tau$-slice of $\zzs{n-1}Q$ and by Theorem \ref{cmpsmashp}, $\zzs{n-1}Q$ is a connected component of $\zZ_{\vv}\tQ$.
Fix a source $v_0=(i_0,0)$ of $Q$ such that $\vu_{v_0}(v)\ge 0$ for any $v$ in $Q$,  a vertex in $Q$ is denoted as $(i,t)$ for $0\le t \le n$, and $t = \vu_{i_0}(i)$.
Then the returning arrow quiver $\tQ$ is obtained from $Q$ by adding an arrow $\beta_p$ from $(\vt(p),\vu_{v(0)}(\vs(p))+n)$ to $(\vs(p),\vu_{v(0)}(\vs(p)))$ for each path $p$ in $\caM$.
We may assume that the second index $t$ is taken from $\zZ/(n+1)\zZ$, so $(i,t) $ and $(i,t+n+1)$ represent the same vertex in $\tQ$.

Let $\ttQ$ be the returning arrow quiver of $\tQ$ defined in \cite{gyz14}, then $\ttQ$ is an $(n+1)$-translation quiver.
$\ttQ$ is obtained from $\tQ$ by adding a loop $\gamma_{i,t}$ to each vertex $(i,t)$ of $Q$.
Since $\tQ$ is a stable $n$-translation quiver, $\ttQ$ is a stable $(n+1)$-translation quiver by \cite{g16}, we can construct their second type $\zZ Q$ construction $\zzv \tQ$ and $\zzv\ttQ$.
$\zzv\tQ$ has $n+1$ connected components isomorphic to $\zzs{n-1} Q$.
$\zzv\tQ$ and $\zzv\ttQ$ have the same vertex set $\{(i,t,s)\mid (i,t)\in Q_0, s\in \zZ\},$ and we have $$t \equiv \vu_{(i_0,0)}(i,t) \bmod (n+1) \mbox{ and } s= \vu_{(i_0,0,r)}(i,t,s)+r,$$ for the vertex $(i,t,s)$ in the connected component of $\zzv\tQ$ containing $(i_0,0,r)$ for some integer $0\le r \le n$.

It follows from the definition that $\zzv\ttQ$ is obtained from $\zzv\tQ$ by adding an arrow $\gamma_{(i,t,s)}$ from $(i,t,s)$ to $(i,t,s+1)$ for each vertex $(i,t,s)$ in $\zzv\ttQ$, which is called {\em an arrow of type $\xc$ from $(i,t,s)$}, connecting the two neighbor connected components of $\zzv\tQ$.

We identify $\zzs{n-1}Q$ with the connected component of $\zzv\tQ$ containing the vertex $v_0=(i_0,0)$.
For $v=(i,s)\in Q_0$, let  $\qv{v}{0,r}$ be the full sub-quiver of $\zZ\mid _{n-1}Q$ with the vertex set $\{(v',t)\in \zZ\mid _{n-1} Q \mid 0\le \vu_{(v,0)}(v',t)\le r\}$.
Then  $\qv{v}{0,n}$ is a complete $\tau$-slice of $\zzs{n-1} Q$ and we call it a {\em homogeneous  $\tau$-slice} with a source $v$.
Write  $V(t) = \{(i,t,t) \in \zzs{n-1}Q_0\} $, then $t=\vu_{i_0}(i)$ and $(i,t,t)\in Q_0\times \zZ/(n+1)\zZ \times \zZ$.
Write $$V'(t,r) = \{(i,t,t+r) \in \zzv\tQ_0\mid  (i, t,t) \in V(t)\},$$ for $t=0, \ldots, n,$ and $r=0, \ldots, n+1$, then $V'(t,0)=V(t)$.
The set $\cup_{t\in \zZ} V'(t,r)$ is the vertex set of a connected component of $\zzv\tQ$ for $r=0, 1,\ldots, n$.
$V'(t,r)$ is in the connected component of $\zzv\tQ$ containing the vertex $v_0=(i_0,0,r)$.
The following lemma describes the arrows in $\zzv\ttQ$.
\begin{lem}\label{zzvttq:arrows}
\begin{enumerate}
\item There is no arrow from $V(t)$ to $V(t')$ in $\zzs{n-1}Q$ unless $t'= t+1$.

\item There is no arrow from $V'(t,r)$ to $V'(t',r')$ in $\zzv\ttQ$ unless $(t',r') = (t+1,r)$ or $(t,r+1)$.
\end{enumerate}
\begin{itemize}
\item[] The arrows from $V'(t,r)$ to $V'(t,r+1)$ are of type $\xc$.
\item[] The arrows from $V'(t,r)$ to $V'(t+1,r)$ are of type $\xa$ or $\xb$, and when $Q$ is a homogeneous $\tau$-slice in $\zzs{n-1}Q$,  they are of type $\xb$ if $t\equiv n\bmod(n+1)$.
\end{itemize}
\end{lem}

We denote the $n$-translation of $\zzs{n-1}Q$ as $\tau$ and the $(n+1)$-translation of $\zzv\ttQ$ as $\ttv$ when we need to distinguish them.

We now describe the $\ttv$-hammock in $\zzv\ttQ$.
For a vertex $(i,t)$ in $Q$, let $H^{(i,t)}$ be the $\tau$-hammock in $\zzs{n-1}Q$ starting at $(i,t,t)$.
Let $H_{\vv}^{(i,t,s)}$ be the $\tau_{\vv}$-hammock in $\zzv\ttQ$  staring at $(i,t,s)$ in $\zzv\ttQ$, then we have the following lemma.

\begin{lem}\label{tower:hammock} The vertex set of
$H_{\vv}^{(i,t,s)}$ is $\{(j,t',s')\mid  (j,t',t') \in H^{(i,t)}, s-t\le s'-t' \le s-t+1 \}$.
\end{lem}
\begin{proof}
Note that for a path $\hat{p}$ in  $\zzv\ttQ$ starting at $(i,t,s)$, there are paths $p_1,\cdots,p_r$ in $Q$ such that $\hat{p} $  is the images of these paths in different copies of the connected components of $\zzv\tQ$ connected by arrows of type $\beta$ and of type $\gamma$.
By Proposition ~\ref{boundpath_rq} and ~\ref{boundpath_zq}, at most one of the arrow of type $\xb$ and at most one of the arrow of type $\xc$ will appear in each bound path in $\zzv\ttQ$.
By Lemma 2.2 of \cite{gyz14} and the definition of $\zzv\ttQ$, using communicative relations concerning the arrows of type $\xc$, if such bound path contains an arrow of type $\xc$, it is linearly dependent to one with the arrow of type $\xc$  at the beginning.
So we may normalize the expressions of the bound paths such that the arrow of type $\xc$ appearing as the first arrow when it appears.

The arrow $\xb_q$ of type $\xb$ start at the vertex $(i',\vu(\vs(q))+n,t)$ for some $q$ in $\caM$, so its position in a normalized bound path is fixed.

Thus when $(p,t,t)$ is a bound  path starting at $(i,t,t)$ and ending at a vertex $(j,t+l,t+l)$ in $\zzs{n-1} Q$, we get exactly two linearly independent bound paths in normalized forms start at $(i,t,s)$, namely,  $(p,t,s)$ containing no arrow of type $\xc$ and ending at $(j,t+l,s+l)$, and  $(p\gamma_{(i,t)},t,s)$ ending at $(j,t+l,s+l+1)$.

So by the definition, $H_{\vv}^{(i,t,s)}$ is exactly the $\tau_{\vv}$-hammock in $\zzv\ttQ$  staring at $(i,t,s)$.
\end{proof}

Note that when forgetting the difference in last indices, $H^{(i,t,s)}_{\vv}$ is just two copies of $H^{(i,t)}$ connected with arrows of type $\gamma$ at every vertices.

We also have a dual version of this Lemma for $\tau_{\vv}$-hammocks ending at a vertex.

\medskip

From the proof of the above Lemma, we get immediately a formula for $\tau_{\vv}$ in $\zzv\ttQ$.
\begin{lem}\label{tower:tauvv}
$\tau^{-1}_{\vv}(i,t,s) = (i, t+n+1 , s+n+2) $.
\end{lem}

Now we can embed the multi-layer quiver $\htQ$ of $Q$ into $\zzv\ttQ$ as the full subquiver with the vertex set $\{(i,t,s)\mid (i,t) \in \Phi_{i_0}(Q)_0, 0\le t < n+1, t \le s \le n+1+t \} $.
We identify $Q_0$ with $\Phi_{i_0}(Q)_0$, that is  use the notation of the image of $\Phi_{i_0}$ as the notation of a vertex in $Q_0$.
The following proposition can be proven directly using Lemma \ref{tower:tauvv}.

\begin{thm}\label{tower:slicevv}
$\htQ$ is a complete $\ttv$-slice in $\zzv\ttQ$.
\end{thm}

Since for an $n$-translation algebra $\tL$ with bound quiver $\tQ$, which is an $n$-translation quiver, its (twisted) trivial extension is an $(n+1)$-translation algebra $\ttL$ with bound quiver $\ttQ$, which is a stable $(n+1)$-translation quiver, by Proposition 4.2 of \cite{g16}.
So by Proposition \ref{0nap:niceg}, $\zZ_{\vv}\ttQ$ is stable $(n+1)$-translation quiver and by Proposition \ref{slice:sub}, the multi-layer quiver $\htQ$, being a complete $\tau_{\vv}$-slice of $\zZ_{\vv}\ttQ$ by the above theorem, is an $(n+1)$-properly-graded quiver.
This gives another proof of Proposition \ref{quiver}.

\begin{thm}\label{tower:np1}
If $Q$ is an $n$-properly-graded quiver, then its multi-layer quiver $\htQ$ is an $(n+1)$-properly-graded quiver.

If $Q$, $\tQ$ and $\ttQ$ are quadratic, then $\htQ^{\perp}$ is an $(n+1)$-slice.
\end{thm}

Starting from a $n$-properly-graded quiver $Q$, by theorem \ref{tower:np1} we have an iterated construction for of higher properly-graded quivers starting at $Q$  depicted  as follows.
\tiny \eqqc{depictingqui}{\xymatrix@C=0.6cm@R=0.6cm{
Q=Q(0)\ar[d]  &                   &Q(1)=\htQ \ar[d]&                  &Q(2)=\widehat{Q(1)}\ar[d] &&\\
\tQ(0)\ar[r]  &\ttQ(1)\ar[d]   &\tQ(1) \ar[r]   &\ttQ(2)\ar[d]  &\tQ(2)\ar@{--}[rr] &&\\
                 &\zZ_\vv\ttQ(1)\ar[ruu] &                   &\zZ_\vv\ttQ(2)\ar[ruu] &&\\
}}\normalsize

We state this picture as the following theorem.

\begin{thm} \label{tower:pgquiver}
Let  $Q$ be a $n$-nicely-graded quiver, then there is a tower $Q(0)=Q, Q(1), \cdots, Q(t), \cdots$ of nicely-graded quivers such that for each $t$,

1. $Q(t+1) = \htQ(t)$ is obtained by multi-quiver construction from $Q(t)$ for all $t$.

2. $Q(t)$ is an $(n+t)$-nicely-graded quiver.

\end{thm}

\section{Algebras and representations associated to multi-layer Quiver\label{algebras}}

Now we consider the algebras associated to the multi-layer quivers of an $n$-nicely-graded quiver.

Let $Q$ be an $n$-nicely-graded quiver and let $\LL$ be the algebra defined by $Q$.
Let $\tQ$ be the returning arrow quiver of $Q$ with relation set $\trho =\rho\cup \rho_{\caM}\cup \{\xb_p\xb_{p'}\mid p,p'\in \caM\}$.
Let $U$ be the $\LL$-bimodules with generators $\{\gamma_v \mid  v\in Q_0\}$ and relation set
$$\{e_{v'} \gamma_v - \delta_{v,v'} \gamma_v, \gamma_v e_{v'}- \delta_{v,v'} \gamma_v\mid  v,v'\in Q_0\}\cup\{ \xa\xc_{\vs(\xa)} -\xc_{\vt(\xa)}\xa \mid \xa\in Q_1\}.$$
Note that $D\LL$ is the $\LL$-bimodule with generators $\{\beta_p \mid  p\in \caM\}$ and relation set $\rho_{0,\caM} \cup  \rho_{\caM}$ where $$\rho_{0,\caM} = \{\xb_pe_{v} -\delta_{\vt(p),v}\xb_p, e_{v}\xb_p-\delta_{\vs(p),v}\xb_p\mid  p\in \caM \}.$$
Let \eqqc{alg}{M(\LL,U,D\LL)= \mat{cccccc}{\LL &0&0&\cdots &0&0\\ U&\LL &0&\cdots &0&0\\ 0 & U&\LL &\cdots &0&0\\ \cdot & \cdot & \cdot  &\cdots &\cdot  &\cdot \\0 &0 &0&\cdots &0&0\\ 0 & 0 &0&\cdots&\LL&0\\ D\LL & 0 &0&\cdots& U&\LL}} be the matrix algebra of $(n+2)\times(n+2)$ matrices defined on $\LL$, $U$ and $D\LL$.

We have the following result.
\begin{pro}\label{multialg}
If $\LL$ is the algebra defined by an $n$-nicely-graded quiver  $Q$, then $M(\LL,U,D\LL)$ is the algebra defined by the multi-layer quiver $\htQ$ of $Q$, that is, if $\LL\simeq kQ/(\rho)$ then $M(\LL,U,D\LL) \simeq k\htQ/(\widehat{\rho})$.
\end{pro}
\begin{proof}
Let $\LL \simeq kQ/(\rho)$, using the same notation for the elements in $kQ$ and their images in $\LL$.
Denote $E_{s,t}$ the matrix unite with the $(s,t)$-entry $1$ and all the other entries $0$.
Define a map \eqqcn{}{\Phi :k \htQ \to M(\LL,U,D\LL)} by \eqqcn{}{\arr{lll}{\Phi(e_{i,t,t+r}) &=& e_{i,t}E_{n+2-r,n+2-r}\\ \Phi(\xa_{i,t,t+r}) &=&\xa_{i,t}E_{n+2-r,n+2-r} \\
\Phi(\xb_{p}) &=&\xb_{p}E_{n+2,1} \\ \Phi(\xc_{i,t,t+r}) &=&\xc_{i,t}E_{n+2-r,n+1-r}.
}}
Then $\Phi$ defines an epimorphism from $k\htQ $ to $M(\LL,U,D\LL)$ and it is easy to see that  $\ker \Phi = (\widehat{\rho})$.
So $M(\LL,U,D\LL)\simeq k\htQ/(\widehat{\rho})$.
\end{proof}

Write $\wht{\LL}=  M(\LL,U,D\LL)$ and we call it {\em the multi-layer algebra of $\LL$}.

\medskip

Write $\wht{S} = \bigoplus_{t=0}^{n+1} \LL\sk{t}$ is a direct sum of $(n+2)$-copies of the algebra $\LL$.
For $0\le t \le n$, let $U\sk{t}$ be the $\LL\sk{t+1}$-$\LL\sk{t}$-bimodule defined on $U$, by setting $\LL\sk{s} U\sk{t} =0$ if $s\neq t+1$, and $  U\sk{t}\LL\sk{s} =0$ if $s\neq t$, $U\sk{t}$ becomes a $\wht{S}$-bimodule.
Let $D\LL\sk{0}$ be the $\LL\sk{n+1}$-$\LL\sk{0}$-bimodule defined on $D\LL$, by setting $\LL\sk{s} D\LL\sk{0} =0$ if $s\neq n+1$, and  $  D\LL\sk{0} \LL\sk{s} =0$ if $s\neq 0$, it becomes a $\wht{S}$-bimodule.
Set $\wht{V} = \sum_{t=0}^{n} U\sk{t} +D\LL\sk{0}$, it is a $\wht{S}$-bimodule.

Then we have obviously the following proposition.

\begin{thm}\label{multialgtensor}
If $\LL$ is $n$-nicely-graded algebra, then $\wht{\LL} \simeq T_{\wht{S}}\wht{V}/(\wht{V}\otimes_{\wht{S}} \wht{V}) $.
\end{thm}

If $\dtL$ is $(p,q)$-Koszul for $p, q\ge 2$, then it is quadratic, so $\LL$ is quadratic, too.
Clearly, we have the following Proposition from the definition.
\begin{pro}\label{qdr} If $\dtL$ is $(p,q)$-Koszul for $p, q\ge 2$, then $\htQ$ is quadratic.
\end{pro}
\begin{proof}
Let $\tL$ be the trivial extension of  $\LL$ and let $\ttL$ be the trivial extension of  $\tL$.
By the proof of Lemma 3.2 of \cite{gyz14}, the simples of $\ttL$ have projective resolution which is linear at first $q$ terms, that is, the $t$th term is generated at degree $t$ for $0 \le t \le q$.

So by Section 3.3 of \cite{bbk02} and Section 2.3 of \cite{bgs96}, $\ttL$ is quadratic.
As the bound quiver of $\ttL$, the quiver $\ttQ$ is quadratic. So $\zzv\ttQ$ is quadratic, by the construction.
Thus $\htQ$ is quadratic, since it is a complete $\tau_\vv$-slice of  $\zzv\ttQ$.
\end{proof}

It is natural to write the quadratic dual of $\htQ$ as $\htQ^{\perp}$ when $\htQ$ is quadratic.

\medskip

Now assume that $Q$ is quadratic and $Q^{\perp}$ is its quadratic dual.
Let $\GG$ be the quadratic dual of $\LL$, then we have that $\GG \simeq kQ/(\rho^{\perp})$.
Let $\htU $ be the $\GG$-bimodule with generators $\{\gamma_v \mid  v\in Q_0\}$ and relation set  $$\{e_{v'} \gamma_v - \delta_{v,v'} \gamma_v, \gamma_v e_{v'}- \delta_{v,v'} \gamma_v\mid  v,v'\in Q_0\}\cup \{\xa\xc_{\vs(\xa)} + \xc_{\vt(\xa)}\xa\mid \xa\in Q_1\}.$$
It follows from Lemma 5.1 of  \cite{g20}, the $\GG$-bimodule generated by $\{\beta_p\mid p\in \caM\}$ with relation $\rho_{0,\caM}\cup {\rho_{\caM}}^{\perp}$ is isomorphic to $\Ext_{\GG}^n(D\GG,\GG)$.
Let $\tilde{U} = \Ext^n_{\GG}(D\GG,\GG)$.
Define the matrix algebra
\eqqc{algg}{ M^*(\GG, \htU,\tilde{U})= \mat{cccccc}{\GG &0&0&\cdots &0&0\\ \htU&\GG &0&\cdots &0&0\\ \htU^{\otimes_{\GG}2} & \htU&\GG &\cdots &0&0\\ \cdot & \cdot & \cdot  &\ddots &\cdot  &\cdot\\ \htU^{\otimes_{\GG}n} & \htU^{\otimes_{\GG}n-1} &\htU^{\otimes_{\GG}n-2}&\cdots&\GG&0\\ \tilde{U} \oplus \htU^{\otimes_{\GG}n+1} & \htU^{\otimes_{\GG}n} &\htU^{\otimes_{\GG}n-1}&\cdots& \htU&\GG}.}
It is obvious that we have the following result.
\begin{thm}\label{qdual}
 $ M^*(\GG, \htU,\tilde{U})=k\htQ/(\hat{\rho}^{\perp})$.
\end{thm}
Write $\htG  =M^*(\GG, \htU,\tilde{U})  $, then $\htG = \wht{\LL}^{!,op}$   and we call it {\em the multi-layer algebra of $\GG = \LL^{!,op}$}.

\medskip

Write $\wht{S}^* = \bigoplus_{t=0}^{n+1} \GG\sk{t}$ is a direct sum of $(n+1)$-copies of the algebra $\GG$.
For $0\le t \le n+1$, let $\htU\sk{t}$ be the $\GG\sk{t+1}$-$\GG\sk{t}$-bimodule defined on $\htU$, by setting $\GG\sk{s} \htU\sk{t} =0$ if $s\neq t+1$, and $  \htU\sk{t}\GG\sk{s} =0$ if $s\neq t$, $\htU\sk{t}$ becomes a $\wht{S}^*$-bimodule.
Let $\tilde{U}\sk{0}$ be the $\GG\sk{0}$-$\GG\sk{n}$-bimodule defined on $\tilde{U}$, by setting $\GG\sk{s} \tilde{U}\sk{0} =0$ if $s\neq 0$, and  $  \tilde{U}\sk{0} \GG\sk{s} =0$ if $s\neq n+1$, it becomes a $\wht{S}^*$-bimodule.
Set $\wht{V}^* = \sum_{t=0}^{n} \htU\sk{t} +\tilde{U}\sk{0}$, it is a $\wht{S}^*$-bimodule.

Then we have obviously the following proposition.

\begin{thm}\label{multidualalgtensor}
If $\GG$ is a nice $n$-slice  algebra, then $\wht{\GG} \simeq T_{\wht{S}^*}\wht{V}^*$.
\end{thm}

Multi-layer algebra construction provided approaches such as representation theory of category to study higher representation theory, as indicating below.

\medskip

We present now the equivalences between the category of modules of the multi-layer algebra and category of the modules of a diagram $\frD$ over the free category $\widetilde{\mathbf A}_{n+1}$ of the quiver $$
\xymatrix@C=2.0cm@R1.5cm{\stackrel{0}{\circ} \ar[r]^{\xc_0} \ar@/_/[rrr]_{\xb} &\stackrel{1}{\circ} \ar@{..}[r] &\stackrel{n}{\circ} \ar[r]^{\xc_n} & \stackrel{n+1}{\circ} }.$$

Let $\caC $ be a skeletal small category, recall that {\em  a strict diagram $\frD $} of $\caC$ is a $2$-functor from $\caC$, naturally view as a $2$-category, to the meta-2-category of locally small abelian categories (see \cite{dlly22b}).
So a strict diagram $\frD $ of $\caC$ consists of the following data:
\begin{enumerate}
\item for any object $i$ in $\mathrm{Ob}(\caC)$, an abelian category $\frD_i$,
\item for any morphism $\xa: i \to j$ in $\mathbf{Mor}(\caC)$, a functor $\frD_{\xa} : \frD_i \to \frD_j$,
\item for any object $i$ in $\mathrm{Ob}(\caC)$, we have $ \frD_{e_i} =\id_{\frD_i}$, and
\item for any pair of composable morphisms $\xa$ and $\xb$ in $\mathrm{Mor}(\caC)$,  $ \frD_{\xb\circ\xa}= \frD_{\xb}\circ \frD_{\xa}.$
\end{enumerate}
We will omit "strict" in the rest of paper, since the diagram of the free category of a quiver, which is used in this paper, is isomorphic to an strict one (Theorem 1.35 of \cite{dlly22b}.

Let $\frD$ be a diagram  of $\caC$.
{\em A (left) module $M$} of the diagram $\frD$ is defined by the following data:
\begin{enumerate}
\item  for any object $i$ in $\Ob(\caC)$, an object $M_i$ in $\frD_i$, and
\item  for any morphism $\xa : i \to j$ in $\Mor(\caC)$, a morphism $M_{\xa} : \frD_{\xa} (M_i) \to M_j$ in $\frD_j$
such that the following two axioms are satisfied.
 \begin{enumerate}
 \item   Given two composable morphisms  $i \stackrel{\xa}{\to} j \stackrel{\xb}{\to} k$ in $\mathrm{Mor}(\caC)$, the diagram
\eqqc{repcomm}{\xymatrix@C=2.0cm@R1.5cm{
\frD_{\xb}\circ\frD_{\xa}(M_i) \ar[r]^{M_{\xb\circ \xa}} \ar@{=}[d]& M_k  \ar@{<-}[d]^{M_{\xb}}\\
\frD_{\xb}(\frD_{\xa} (M_i)) \ar[r]^{\frD_{\xb}(M_{\xa})} &\frD_{\xb}(M_j)
}
}
commutes, that is,  $M_{\xb\circ \xa} = M_{\xb}D_{\xb}(M_{\xa})$

 \item Given any object $i$ in $\Ob(\caC)$, the diagram
\eqqc{repcomm2}{\xymatrix@C=2.0cm@R1.5cm{
M_i \ar[rd]^{\id_{\frD_{e_i}}} \ar[rr]^{\id_{M_i}}
&&M_i\\
&\frD_{e_i}(M_i) \ar[ru]^{M_{e_i}} &
}
}
commutes, that is,  $\id_{M_i} =M_{e_i}\id_{\frD_{e_i}}$.
 \end{enumerate}
\end{enumerate}

A morphism $M\stackrel{\omega}{\to} M'$ between two modules of the diagram $\frD$ is  a family $\{\omega_i: M_i\to M'_i\}_{i\in \Ob(\caC)}$ of morphisms such that the diagram
\eqqc{morcomm}{\xymatrix@C=2.0cm@R1.5cm{
\frD_{\xa}(M_i) \ar[r]^{\frD_{\xa}(\omega_i)} \ar[d]^{M_{\xa}}& \frD_{\xa}(M'_i) \ar[d]^{M'_{\xa}}\\
M_j \ar[r]^{\omega_j} & M'_j
}
}
commutes for any morphism $\xa : i\to j$ in $\Mor(\caC)$.
Write $\rMod(\caC,\frD)$ for the category of the modules of the diagram $\frD$, it is an additive category which has a zero object, the object with zero in each component.

Given an $n$-properly-graded algebra $\LL$, the diagram $\dbA{n+1}{\LL}$ ($\dbfA{n+1}{\LL}$, respectively) is defined as the strict diagram assigning to each vertex $i$ the category $\rMod \LL$ of left $\LL$-modules (the $\rmod \LL$ of finitely generated left $\LL$-modules, respectively), assigning to each arrow $\xc_t$ the functor  $U \otimes_{\LL} \mbox{-} $ and assigning to the arrow the arrow $\xb$ the functor $D\LL \otimes_{\LL} \mbox{-} $.
Write $\rMod(\dbA{n+1}{\LL})$  ($\rMod(\dbfA{n+1}{\LL})$, respectively) for the module category for the diagram $\dbA{n+1}{\LL}$ ($\dbfA{n+1}{\LL}$, respectively).

By setting $M_{\xb\circ\xa} =0$ for any compasable arrows $i \stackrel{\xa}{\to} j \stackrel{\xb}{\to} k$, we get a bound module category $\rMod^b(\dbA{n+1}{\LL})$ and $\rMod^b(\dbfA{n+1}{\LL})$ of these diagrams, respectively.

Given an $n$-slice algebra $\GG$, we define the diagram $\dbA{n+1}{\GG}$ ($\dbfA{n+1}{\GG}$) to be the strict diagram assigning to each vertex $i$ the category $\rMod \GG$ of left $\GG$-modules ($\rmod \GG$ of finitely generated left $\GG$-modules), assigning to each arrow $\xc_t$ the functor  $\htU \otimes_{\GG} \mbox{-} $ and assigning to the arrow $\xb$ the functor $\tilde{U} \otimes_{\GG} \mbox{-} $ of these diagrams, respectively.

Note that $1= \sum_{r=0}^{n+1} E_{r,r}$ in $\htL$ and $ E_{r,r}\htL E_{r,r} \simeq \GG$, we regard this isomorphism as an identity.
Let $\htM$ be a $\htL$-module, write $M_r =E_{n+1-r,n+1-r} \htM$, then $M$ is a $\LL$-module, and as the $M(\LL,U,D\LL)$-module, we have $$\htM = \mat{c}{M_{n+1}\\ \vdots \\ M_1\\ M_0},$$  the operation of $M(\LL,U,D\LL)$ on $\htM$ induces morphisms $f_r:U \otimes M_{r} \to M_{r+1}$ for $r=0,\ldots,n$ satisfying $f_{r+1}f_r=0$ and $g: D\LL \otimes M_0 \to M_{n+1}$ of $\LL$-modules.
By assigning to vertex $r$ the $\LL$-module $M_r$ for $r=0,1,\ldots,n+1$, assigning to the arrow $\xc_r$ the morphism $f_r$ for $r=0,1,\ldots,n$ and  assigning to the arrow $\xb$ the morphism $g$, a module of the diagram $\dbA{n+1}{\LL}$ or $\dbfA{n+1}{\LL}$ of quiver $\wtt{\mathbf A}_{n+1}$ is obtained.
One checks directly that this induced equivalences between $\rMod \wht{\LL}$ and $\rMod^b(\dbA{n+1}{\LL})$, and between $\rmod \wht{\LL}$ and $\rMod^b(\dbfA{n+1}{\LL})$.
So we have the following theorem for the multi-layer algebra of an $n$-properly-graded algebra.

\begin{thm}\label{eqmod}Let $\LL$ be a nicely-graded $n$-properly-graded algebra.
Then the module category $\rMod^b(\dbA{n+1}{\LL})$ is equivalent to the category $\rMod \wht{\LL}$ of $\wht{\LL}$ modules and $\rMod^b(\dbfA{n+1}{\LL})$ is equivalent to  $\rmod \wht{\LL}$ of finitely generated $\wht{\LL}$-modules.
\end{thm}

Similarly, we have the following theorem for the multi-layer algebra of an $n$-slice algebra.

\begin{thm}\label{eqmods}Let $\GG$ be a nicely-graded  $n$-slice algebra.
Then the module category $\rMod(\dbA{n+1}{\GG})$ is equivalent to the category $\rMod \wht{\GG}$ of $\wht{\GG}$-modules and $\rMod(\dbfA{n+1}{\GG})$ is equivalent to  $\rmod \wht{\GG}$ of finitely generated $\wht{\GG}$-modules.
\end{thm}

It is interesting to study the representation theory using the category of the modules of the diagram.
The modules of the diagram is formed by a sequence $(M_{0},\ldots, M_{n+1}, f_0,\cdots,f_n,g)$ of $\LL$(or $\GG$)-modules $M_{0},\ldots, M_{n+1}$ and morphisms $f_0,\cdots,f_n, g$, where $f_t: M_{t+1} \to M_t$ are $\LL$(or $\GG$)-morphisms satisfying certain conditions.

For the case of $n$-properly-graded algebra, such sequences can either be regarded as pairs $(M^{\bullet}, g)_{[0,n+1]}$ of certain complex $M^{\bullet}$ of length $n+2$ concentrated in the interval $[0,n+1]$ in the category of complexes $\mathcal C(\LL)$ of $\LL$-modules with a map $g$ connect its end term, or as pairs of representation $(\tilde{M},g)$ of the the quiver $\mathbf A_{n+1}$ of type $ A_{n+1}$ with linear order over the algebra $\LL$ with a $\LL$-map from the last one to the first one.
The representation theory of $\htL$ and $\htG$ are related to the morphism categories $\mathrm{Mor}_n(\LL)$ and $\mathrm{Mor}_n(\GG)$ (see \cite{zp11,gk22}).

We leave further study of the representation theory of multi-layer quivers and algebras in the future.

\section{Higher slice algebra of infinite type\label{inftype}}

When we start with an $n$-slice algebra $\GG$ with a nice $n$-slice quiver $Q^{\perp}$, we obtain an $n$-nicely-graded quiver $Q$ with a $n$-properly-graded algebra $\LL=\GG^{!,op}$  whose trivial extension $\tL$ is an $n$-translation algebra.
Then we construct the multi-layer quiver for $Q$ to get an $(n+1)$-nicely-graded quiver $\htQ$, and $(n+1)$-slice $\htQ^{\perp}$ and the multi-layer algebra $\wht{\GG}$ defined by $\htQ^{\perp}$.

For an $n$-slice algebra $\GG$ of finite type, the trivial extension of $\wht{\GG^{!,op}}$ is an $(n+1)$-properly-graded algebra, the following example shows that it may not be an $(n+1)$-translation algebra.
So the multi-layer algebra of an $n$-slice algebra of finite type may not be an $(n+1)$-slice algebra of finite type.

\begin{exa}\label{exa:noslicealg}
{\em
Start with the quiver
$$
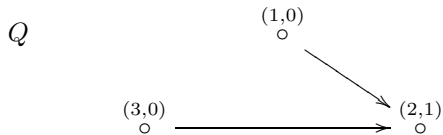
\xymatrix@C=1.0cm@R0.5cm{
&Q&
&\stackrel{(1,0)}{\circ}\ar[rd]&\\
&&\stackrel{(3,0)}{\circ}\ar[rr]&&\stackrel{(2,1)}{\circ}
}
$$
of type $A_3$.
We have $Q$ is a nicely-graded $1$-properly-graded quiver and  $Q^{\perp}=Q$ is a nicely-graded $1$-slice  quiver.

Let $\LL$ be the algebra defined by $Q$, then $\LL$ is a $1$-properly-graded algebra, its trivial extension $\tL$ an $1$-translation algebra by \cite{bbk02}.
So $\GG=\LL^{!,op}=\LL$ is an $1$-slice algebra.
Construct its multi-layer quiver $\htQ$ and its returning arrow quiver $\wtt{\htQ}$ as follow.
$$
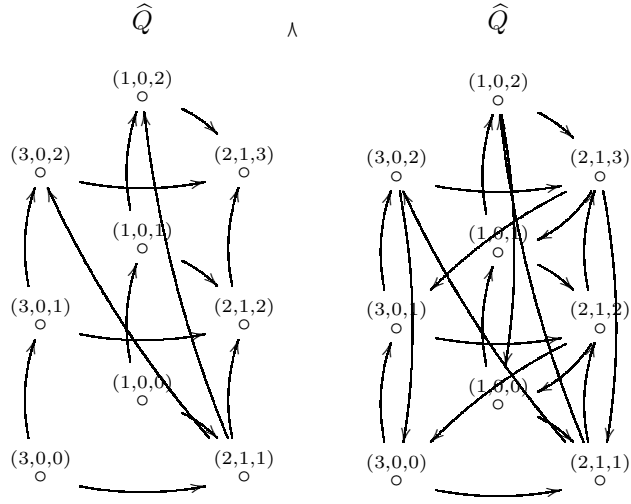
\xymatrix@C=0.3cm@R0.3cm{
&\htQ&
\\
&\stackrel{(1,0,2)}{\circ}\ar@/^/[rd]&\\
\stackrel{(3,0,2)}{\circ}\ar@/_/[rr]& &\stackrel{(2,1,3)}{\circ}\\
&\stackrel{(1,0,1)}{\circ}\ar@/^/[rd]\ar@[purple]@/^/[uu]&\\
\stackrel{(3,0,1)}{\circ}\ar@/_/[rr]\ar@[purple]@/^/[uu]& &\stackrel{(2,1,2)}{\circ}\ar@[purple]@/^/[uu]\\
&\stackrel{(1,0,0)}{\circ}\ar@/^/[rd]\ar@[purple]@/^/[uu]&\\
\stackrel{(3,0,0)}{\circ}\ar@/_/[rr] \ar@[purple]@/^/[uu]& &\stackrel{(2,1,1)}{\circ} \ar@[purple]@/^/[uu] \ar@[green]@/^/[luuuuu]\ar@[green]@/^/[lluuuu]\\
\\
}
\xymatrix@C=0.5cm@R0.5cm{\ar@[white]&}
\xymatrix@C=0.3cm@R0.3cm{
&\wtt{\htQ}&
\\
&\stackrel{(1,0,2)}{\circ}\ar@/^/[rd] \ar@[purple]@/^/[dddd]&\\
\stackrel{(3,0,2)}{\circ}\ar@/_/[rr] \ar@[purple]@/^/[dddd]& &\stackrel{(2,1,3)}{\circ} \ar@[purple]@/^/[dddd] \ar@[green]@/_/[ddll]\ar@[green]@/^/[dl]\\
&\stackrel{(1,0,1)}{\circ}\ar@/^/[rd]\ar@[purple]@/^/[uu]&\\
\stackrel{(3,0,1)}{\circ}\ar@/_/[rr]\ar@[purple]@/^/[uu]& &\stackrel{(2,1,2)}{\circ} \ar@[purple]@/^/[uu] \ar@[green]@/_/[ddll]\ar@[green]@/^/[dl]\\
&\stackrel{(1,0,0)}{\circ}\ar@/^/[rd]\ar@[purple]@/^/[uu]&\\
\stackrel{(3,0,0)}{\circ}\ar@/_/[rr] \ar@[purple]@/^/[uu]& &\stackrel{(2,1,1)}{\circ} \ar@[purple]@/^/[uu] \ar@[green]@/^/[lluuuu]\ar@[green]@/^/[luuuuu]\\
\\
}
$$

$\htQ$ is a $2$-properly-graded quiver, and $\wtt{\htQ}$ is a stable $2$-translation quiver.

By indexing the vertices as $(i,t,r)$ with  $r\in \zZ/3\zZ$, the arrows are written as
\eqqcn{ex:arrow}{\arr{lcl}{\xa_{i,t,r}:& (i,t,r)\to (2,t+1,r+1)& i\neq 2\\ \xb_{j,t,r}: & (2,t,r) \to (j,t-1,r+1) &j=1,3\\
\xc_{i,t,r}:& (i,t,r) \to (i,t,r+1)&
}}
and we write them in short as $\xa,\xb_1,\xb_3,\xc$ for short.
The relations become
\eqqcn{eq:relations}{\{\xc\xc \}\cup\{\xc\xa-\xa\xc,\xc\xb_1-\xb_1\xc,\xc\xb_3-\xb_3\xc\}.}

Let $\wht{\LL}$ be the algebra defined by the $2$-properly-graded quiver $\htQ$ and let $\wtt{\wht{\LL}}$ be the algebra defined $\wtt{\htQ}$.
$\wht{\LL}$ is a $2$-properly-graded algebra and  $\wtt{\wht{\LL}}$ is its trivial extension.
The indecomposable projective-injective $\wtt{\wht{\LL}}$-modules look like
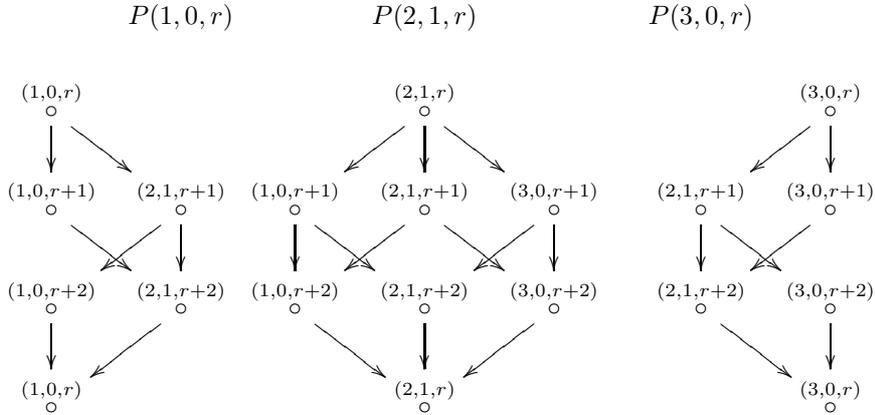
\eqqcn{p100}{
\xymatrix@C=0.2cm@R0.6cm{ &P({1,0,r})\\
\stackrel{(1,0,r)}{\circ}\ar@[purple][d] \ar[rd]&\\
\stackrel{(1,0,r+1)}{\circ}\ar[rd]& \stackrel{(2,1,r+1)}{\circ}\ar@[green][ld] \ar@[purple][d] \\
\stackrel{(1,0,r+2)}{\circ}\ar@[purple][d]& \stackrel{(2,1,r+2)}{\circ}\ar@[green][ld]\\
\stackrel{(1,0,r)}{\circ}&\\
}
\xymatrix@C=0.2cm@R0.6cm{ &P({2,1,r})\\
&\stackrel{(2,1,r)}{\circ} \ar@[green][rd] \ar@[green][ld] \ar@[purple][d]&\\
\stackrel{(1,0,r+1)}{\circ}\ar@[purple][d] \ar[rd] & \stackrel{(2,1,r+1)}{\circ} \ar@[green][ld] \ar@[green][rd] & \stackrel{(3,0,r+1)}{\circ} \ar@[purple][d] \ar[ld]\\
\stackrel{(1,0,r+2)}{\circ} \ar[rd] & \stackrel{(2,1,r+2)}{\circ} \ar@[purple][d] & \stackrel{(3,0,r+2)}{\circ} \ar[ld]\\
&\stackrel{(2,1,r)}{\circ} &\\
}
\xymatrix@C=0.2cm@R0.6cm{ &P({3,0,r})\\
&&\stackrel{(3,0,r)}{\circ}\ar@[purple][d] \ar[ld]&\\
&\stackrel{(2,1,r+1)}{\circ}\ar@[green][rd] \ar@[purple][d]& \stackrel{(3,0,r+1)}{\circ}\ar[ld]\\
&\stackrel{(2,1,r+2)}{\circ}\ar@[green][rd]& \stackrel{(3,0,r+2)}{\circ}\ar@[purple][d]\\
&&\stackrel{(3,0,r)}{\circ}\\
}
}

Consider the projective resolution
\eqqc{prresS100}{\cdots \slrw{f^2} P^1\slrw{f^1} P^0 \lrw S(1,0,0)\lrw 0 }
of simple $S=S(1,0,0)$ corresponding to the vertex $(1,0,0)$.
Note that $\om{}S$ is generated as submodule by $\xc e_{1,0,0}$ and $\xa e_{1,0,0}$ in degree $1$, so $P^1= P(1,0,1)\oplus P(2,1,1)$ with $f(e_{1,0,1}) =\xc e_{1,0,0}$ and $f(e_{2,1,1}) =\xa e_{1,0,0}$.

We have $f(\xc e_{1,0,1}) =0$, $f(\xa e_{1,0,1}) =f(\xc e_{2,1,1})$ and $f(\xb_3 e_{2,1,1}) =0$.
So the element $\xc e_{1,0,1}, \xa e_{1,0,1}- \xc e_{2,1,1}, \xb_3 e_{2,1,1}$ are in the kernel of $f_2$,  they are elements of degree $2$ and generate $ \om{2}S$.

Now $P^2= P(1,0,2)\oplus P(2,1,2)\oplus P(3,0,2)$ with $f(e_{1,0,2}) =\xc e_{1,0,1}$, $f(e_{2,1,2}) =\xa e_{1,0,1}-\xc e_{2,1,1}$ and $f(e_{3,0,2}) =\xb_3 e_{2,1,1}$.
Then $f(\xc e_{1,0,2}) =0$, $$f(\xa e_{1,0,2}) = \xa\xc e_{1,0,1} = \xc\xa e_{1,0,1} =\xc \xa e_{1,0,1}-\xc e_{2,1,1} = f(\xc e_{2,1,2})$$ and $f(\xc e_{3,0,2}) =0$.
So $\xc e_{1,0,2},\xa e_{1,0,2}-\xc e_{2,1,2},\xc e_{3,0,2}$ are linearly independent degree $3$ elements in $\ker f_2$.
Note that $e_{3,0,1} \om{2}S=0$, thus $\xc\xb_3 e_{2,1,2}, \xb_3\xa e_{3,0,2}$ are both degree $4$ elements in $\ker f_3$.
It is easy to see that $\xb_3\xa e_{3,0,2}$ is not in the submodule generated by the elements $\xc e_{1,0,2},\xa e_{1,0,2}-\xc e_{2,1,2},\xc e_{3,0,2}$.
This shows that $\om{3}S$ has at least one generator of degree $4$ and $3$ generators of degree $3$.

Thus the resolution \eqref{prresS100} is not linear at $P^3$ and $\om{3}S$ is not semisimple.
So $\wtt{\wht{\LL}}$ is not almost Koszul, and $\wht{\LL}^{!,op}$ is not a $2$-slice algebra.
}\end{exa}

This example shows that the multi-layer algebra of an $n$-slice algebra of finite type may not be an $n$-slice algebra.
It is interesting to know if such algebra is $(n+1)$-hereditary?

\medskip

The following theorem indicating that the multi-layer algebra of an $n$-slice algebra of infinite type is an $(n+1)$-slice algebra of infinite type.

\begin{thm}\label{alg:inf}
Let $\GG$ be a nicely-graded $n$-slice algebra of infinite type, then $\wht{\GG}$ is a nicely-graded $(n+1)$-slice algebra of infinite type.
\end{thm}
\begin{proof}
Let $\LL =\GG^{!,op}$ be the quadratic dual of $\GG$.
Let $Q^{\perp}$ be the bound quiver of $\GG$, and  $Q$ be the bound quiver of $\LL$, then the trivial extension $\dtL$ is a Koszul stable $n$-translation algebra, that is, a Koszul self-injective  algebra of Loewy length $n+1$.
By Corollary 3.3 of \cite{gyz14}, the trivial extension $\Delta \dtL$ of $\dtL$ is a Koszul self-injective  algebra of Loewy length $n+2$, or, a Koszul $(n+1)$-translation algebra.
By the proof of Theorem 5.3 of \cite{g16}, $(\Delta \dtL)\# k \mathbb Z$ is also Koszul $(n+1)$-translation algebra.

In this case, $Q^{\perp}$ is a nice $n$-slice and $Q$ is nice $n$-properly-graded quiver.
The bound quiver of $\dtL$ is the returning arrow quiver $\tQ$, and the bound quiver of $\Delta \dtL$ is the returning arrow quiver $\ttQ$ of $\tQ$.
But we have that the bound quiver of $(\Delta \dtL)\# k \mathbb Z$ is $\zzv\ttQ$.
Note that $\htQ $ is a complete $\tau$-slice in $\zzs{n} \htQ$ and we have an isomorphism $\zzs{n} {\htQ} \simeq \zzv\ttQ$.
$\zzs{n} {\htQ}$ is the bound quiver of $\wtt{\htL}\# k \mathbb Z $.
This implies that $\wtt{\htL}\# k \mathbb Z$ is a Koszul $(n+1)$-translation algebra, so $\wtt{\htL}$ is also a Koszul $(n+1)$-translation algebra by Theorem 5.3 of \cite{g16}.
By Theorem \ref{tower:np1}, $\htQ^{\perp}$ is a nicely-graded $(n+1)$-slice, and by Theorem 6.6 of \cite{g20} and Proposition 2.2.1 and 2.9.1 of \cite{bgs96}, the $(n+2)$-preprojective algebra of $\htG$ is also Koszul, so $\wht{\GG}$ is an $(n+1)$-slice algebra of infinite type.
\end{proof}

We remark that in this case, Theorem \ref{tower:hammock} describes the $(n+1)$-almost split sequences in the $(n+1)$-preprojective and $(n+1)$-preinjective components of $\htG$, by Theorem 5.6 of \cite{gll19b}.

\medskip

The Picture \eqref{depictingqui} in Section \ref{mqtslice} also applied for the algebras $\LL$, $\tL$, $\ttL$, $\ttL\#\mathbb Z^*$  and $\htL$, and for their quadratic duals  $\GG$, $\tG$, $\ttG$, $(\ttL\#\mathbb Z^*)^{!,op}$ and $\htG$.

Start with a $n$-slice algebra $\GG$ of infinite type $\GG$, we have the following analog of Theorem \ref{tower:pgquiver}.
The algebra $\GG(t)$ constructed in this way is an $(n+1)$-slice algebra of infinite type by Theorem \ref{alg:inf}.
By this way, we get a tower of the higher slice algebras of infinite type as presented below.
The last assertion follows from Theorem 5.10 of \cite{gll19b}.

\begin{thm} \label{tower:slicealgebra}
Let  $\GG$ be a nice  $n$-slice algebra of infinite type with bound quiver $Q^{\perp}(0)= Q^{\perp}$, then there is a tower $\GG(0)=\GG, \GG(1), \cdots, \GG(t), \cdots$ of nicely-graded higher slice algebras such that for each $t$,

1. $\GG(t+1) = \htG(t)$ is obtained by multi-quiver construction from $\GG(t)$ for all $t$.

2. $\GG(t)$ is a nice $(n+t)$-slice algebra of infinite type with bound quiver $Q^{\perp}(t) =\htQ(t-1)^{\perp}$.

3. $\GG(t)$ is both a subalgebra and also a factor algebra of $\GG(t+1)$ for all $t$.

4. The category $\rMod \GG(t)$ of $\GG(t)$-modules is equivalent to the module category $\rMod(\dbA{n+t+1}{\GG(t)})$  of the diagram $\dbA{n+t+1}{\GG(t)}$, and the module category $\rmod \GG(t)$ of finitely generated $\GG(t)$-modules is equivalent to the module category $\rMod(\dbfA{n+t+1}{\GG(t)})$ of the diagram $\dbfA{n+t+1}{\GG(t)}$.

5. The Auslander-Reiten quivers of the $(n+t)$-preprojective component and of the $(n+t)$-preprojective component of $\GG(t)$ are truncations of $\zzs{n+t-1} Q(t)^{op,\perp}$.
\end{thm}

Recall that a $n$-slice algebra of infinite type is an $n$-representation-infinite algebra, its natural to ask if we have a similar construction got higher representation-infinite algebras in general.

\medskip

By Theorem 2.10 of \cite{hio12}, the tensor product of higher representation-infinite algebras is a higher representation-infinite algebra.
But this is not the case for higher slice algebras of infinite type, as is shown in the following example.
\begin{exa}\label{exa:tensornoslice}{\em
Let $\GG$ be the $1$-slice algebra defined by the quiver
\eqqcn{}{\xymatrix@C=0.2cm@R0.5cm{
1\ar@/^/[rr]^{\xa}\ar@/_/[rr]_{\xb} && 2
},}
let $\check{\GG} = \GG \otimes \GG$ be the tensor algebra of $\GG$ with $\GG$.
$\GG$ is representation infinite hereditary algebra, hence it is $1$-representation-infinite.
Then by Theorem 2.10 of \cite{hio12}, $\check{\GG}$ is $2$-representation-infinite.

Note $\GG$ has a basis $ \{e_1,e_2,\xa,\xb\}$, $\check{\GG}$ has a basis , $ \{e_1\otimes e_1, e_1\otimes e_2 ,e_2\otimes e_1, e_2\otimes e_2 ,e_1\otimes \xa,e_1\otimes \xb ,e_2\otimes \xa,e_2\otimes \xb, \xa \otimes e_1, \xb \otimes e_1, \xa \otimes e_2, \xb \otimes e_2 ,\xa\otimes \xa, \xb\otimes \xb, \xa\otimes \xb, \xb\otimes \xa\}$ with \eqqcn{}{\arr{l}{e_1\otimes e_1+ e_1\otimes e_2 +e_2\otimes e_1+ e_2\otimes e_2 =1\\
(e_s\otimes e_t)( e_{s'}\otimes e_{t'}) = \delta_{s,s'} \delta_{t,t'}e_s\otimes e_t\\}}
Write $e_{s,t}= e_s\otimes e_t$ write $\xa_{u} =e_1\otimes \xa, \xb_u= e_1\otimes \xb , \xa_d = e_2\otimes \xa, \xb_d= e_2\otimes \xb, \xa_l = \xa \otimes e_1, \xb_l = \xb \otimes e_1, \xa_r= \xa \otimes e_2, \xb_r= \xb \otimes e_2  $ so $\check{\GG}$ is an algebra with quiver \eqqcn{}{\check{Q} \xymatrix@C=0.8cm@R0.5cm{
(1,1)\ar@/^/[rr]^{\xa_u}\ar@/_/[rr]_{\xb_u} \ar@/^/[d]^{\xa_l}\ar@/_/[d]_{\xb_l}&& (1,2) \ar@/^/[d]^{\xa_r}\ar@/_/[d]_{\xb_r}\\(2,1)\ar@/^/[rr]^{\xa_d}\ar@/_/[rr]_{\xb_d} && (2,2)
},}
with relations \eqqcn{}{\xa_r\xa_u -\xa_d\xa_l, \xb_r\xb_u -\xb_d\xb_l, \xa_d\xb_l-\xb_d\xa_u, \xb_d\xa_l-\xa_d\xb_l.}
So its quadratic dual quiver $\check{Q} = (\check{Q}_0,\check{Q}_1,\check{\rho})$ has the relation set  \eqqcn{}{\check{\rho}=\{\xa_r\xa_u +\xa_d\xa_l, \xb_r\xb_u +\xb_d\xb_l, \xa_d\xb_l+\xb_d\xa_u, \xb_d\xa_l+\xa_d\xb_l\}.}
The returning arrow quiver is \eqqcn{}{\wtb{Q} 
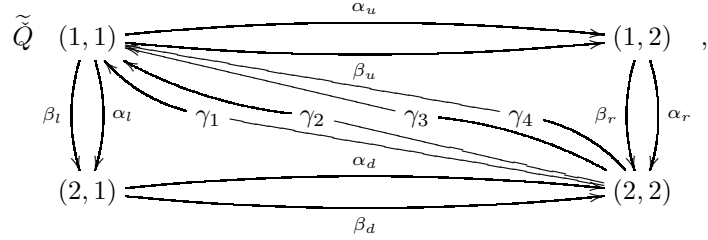
\xymatrix@C=0.8cm@R0.5cm{
(1,1)\ar@/^/[rrrrr]^{\xa_u}\ar@/_/[rrrrr]_{\xb_u} \ar@/^/[dd]^{\xa_l}\ar@/_/[dd]_{\xb_l} &&&&& (1,2) \ar@/^/[dd]^{\xa_r}\ar@/_/[dd]_{\xb_r}\\ & \xc_1\ar@/^/[ul]& \xc_2\ar@/^/[ull]& \xc_3\ar[ulll]& \xc_4\ar[ullll]&\\ (2,1)\ar@/^/[rrrrr]^{\xa_d}\ar@/_/[rrrrr]_{\xb_d} &&&&& (2,2)\ar@{-}@/_/[ul]\ar@{-}@/_/[ull]\ar@{-}[ulll]\ar@{-}[ullll]
},}
with $4$ returning arrows $\xc_1= (\xa_r\xa_u)^*,\xc_2= (\xb_r\xb_u)^*,\xc_3= (\xa_d\xb_l)^*$ and $\xc_4= (\xb_d\xa_l)^*$ from $(2,2)$ to $(1,1)$.
The relations are elements in \eqqcn{l}{\arr{l}{R = \{\xa_r\xa_u +\xa_d\xa_l, \xb_r\xb_u +\xb_d\xb_l, \xa_d\xb_l+\xb_d\xa_u, \xb_d\xa_l+\xa_d\xb_l, \\ \quad \xc_1\xb_r, \xc_1\xb_d, \xc_2\xa_r, \xc_2\xa_d,\xb_l\xc_1,\xb_u\xc_1, \xa_l\xc_2,\xa_u\xc_2, \\ \quad \xc_3\xb_d, \xc_3\xa_r, \xc_4\xa_d, \xc_4\xb_r, \xa_l\xc_3, \xb_u\xc_3, \xb_l\xc_4, \xa_u\xc_4\} ,}} and the elements in the set $ \wtb{Q}_4$ of paths of length $4$.
Note that  $\wtb{Q}_4 \not \subset (R)$ and thus the returning arrow quiver $\wtb{Q}$ is not a quadratic one.
This implies that $\check{\GG}$ is not an $2$-slice algebra.
}\end{exa}

{}

\end{document}